\documentclass[11pt]{amsart}
\usepackage{
	fullpage,
	amssymb,
	amsmath,
	hyperref,
	color,
	enumitem,
	caption,
	comment,
	stmaryrd,
	tabularx,
	textcomp}
\usepackage[all,pdf]{xy}
\usepackage[noadjust]{cite}
\usepackage[abs]{overpic}
\usepackage{dashbox}

\newcolumntype{L}{>{\raggedright\arraybackslash}m{.3105\textwidth}}
\newcolumntype{M}{>{\raggedright\arraybackslash}m{.477\textwidth}}
\newcolumntype{W}{>{\centering\arraybackslash}m{.477\linewidth}}
\newcolumntype{H}{>{\centering\arraybackslash}m{.3105\textwidth}}

\newtheorem{thm}{Theorem}[section]
\newtheorem{lem}[thm]{Lemma}
\newtheorem{prop}[thm]{Proposition}
\newtheorem{cor}[thm]{Corollary}
\newtheorem{conj}[thm]{Conjecture}

\theoremstyle{definition}
\newtheorem{defn}[thm]{Definition}
\newtheorem{ex}[thm]{Example}

\newtheorem{rem}[thm]{Remark}

\newcommand{\bbA}{\mathbb{A}}
\newcommand{\bbC}{\mathbb{C}}
\newcommand{\bbF}{\mathbb{F}}
\newcommand{\bbN}{\mathbb{N}}
\newcommand{\bbP}{\mathbb{P}}
\newcommand{\bbQ}{\mathbb{Q}}

\newcommand{\bbZ}{\mathbb{Z}}

\newcommand{\bsJ}{\boldsymbol{J}}
\newcommand{\bsP}{\boldsymbol{P}}

\newcommand{\calD}{\mathcal{D}}
\newcommand{\calE}{\mathcal{E}}

\newcommand{\calK}{\mathcal{K}}
\newcommand{\calL}{\mathcal{L}}

\newcommand{\calO}{\mathcal{O}}
\newcommand{\calP}{\mathcal{P}}
\newcommand{\calQ}{\mathcal{Q}}
\newcommand{\calR}{\mathcal{R}}
\newcommand{\calS}{\mathcal{S}}
\newcommand{\calU}{\mathcal{U}}

\newcommand{\OK}{\calO_K}

\newcommand{\Kbar}{\overline{K}}

\newcommand{\QQbar}{\overline{\bbQ}}

\renewcommand{\hbar}{\overline{h}}

\newcommand{\frakp}{\mathfrak{p}}

\newcommand{\ord}{\operatorname{ord}}
\newcommand{\Gal}{\operatorname{Gal}}

\newcommand{\coker}{\operatorname{coker}}

\newcommand{\Spec}{\operatorname{Spec}}
\newcommand{\tors}{{\operatorname{tors}}}

\newcommand{\Jac}{\operatorname{Jac}}
\newcommand{\Res}{\operatorname{Res}}

\newcommand{\PrePer}{\operatorname{PrePer}}

\newcommand{\disc}{\operatorname{disc}}

\newcommand{\rk}{\operatorname{rk}}

\newcommand{\Ell}{\operatorname{ell}}
\newcommand{\Jell}{J^{\Ell}}
\newcommand{\Xell}{X^{\Ell}}
\newcommand{\Yell}{Y^{\Ell}}

\newcommand{\Ann}{\operatorname{Ann}}
\newcommand{\Log}{\operatorname{Log}}

\newcommand{\bsa}{{\boldsymbol{a}}}

\newcommand{\bsf}{{\boldsymbol{f}}}

\renewcommand{\tilde}{\widetilde}
\renewcommand{\phi}{\varphi}

\newcommand{\longto}{\longrightarrow}
\newcommand{\longhookrightarrow}{\lhook\joinrel\longrightarrow}

\newcommand{\Mu}{\textup{M}}

\newcommand{\pic}{\includegraphics[scale=.6]}

\renewcommand{\labelenumi}{(\Alph{enumi})}

\usepackage{verbatim}

\numberwithin{equation}{section}

\title[Preperiodic points over cyclotomic quadratic fields]{Preperiodic points for quadratic polynomials over cyclotomic quadratic fields}
\author{John R. Doyle}
\address{Department of Mathematics and Statistics\\
Louisiana Tech University\\
Ruston, LA 71272} 
\email{jdoyle@latech.edu}

\begin{document}

\begin{abstract}
Given a number field $K$ and a polynomial $f(z) \in K[z]$ of degree at least 2, one can construct a finite directed graph $G(f,K)$ whose vertices are the $K$-rational preperiodic points for $f$, with an edge $\alpha \to \beta$ if and only if $f(\alpha) = \beta$. Restricting to quadratic polynomials, the dynamical uniform boundedness conjecture of Morton and Silverman suggests that for a given number field $K$, there should only be finitely many isomorphism classes of directed graphs that arise in this way. Poonen has given a conjecturally complete classification of all such directed graphs over $\bbQ$, while recent work of the author, Faber, and Krumm has provided a detailed study of this question for all quadratic extensions of $\bbQ$. In this article, we give a conjecturally complete classification like Poonen's, but over the cyclotomic quadratic fields $\bbQ(\sqrt{-1})$ and $\bbQ(\sqrt{-3})$. The main tools we use are \emph{dynamical modular curves} and results concerning quadratic points on curves.
\end{abstract}



\maketitle

\section{Introduction}
Let $K$ be a number field, and let $f \in K(z)$ be a rational map of degree $d \ge 2$. For each integer $n \ge 0$, we let $f^n$ denote the $n$-fold composition of $f$; that is, $f^0$ is the identity, and $f^n = f \circ f^{n-1}$ for each $n \ge 1$. We say that $\alpha \in \bbP^1(K)$ is {\bf preperiodic} for $f$ if there exist integers $m \ge 0$ and $n \ge 1$ such that $f^{m + n}(\alpha) = f^m(\alpha)$; in this case, the minimal such $m$ and $n$ are called the {\bf preperiod} and {\bf eventual period}, respectively, and we refer to the pair $(m,n)$ as the {\bf (preperiodic) portrait} of $\alpha$. If the preperiod is 0, we say that $\alpha$ is {\bf periodic} with {\bf period} $n$. We set
	\[
		\PrePer(f,K) := \{\alpha \in \bbP^1(K) : \alpha \text{ is preperiodic for } f\}.
	\]
We denote by $G(f,K)$ the functional graph associated to the restriction of $f$ to $\PrePer(f,K)$; that is, the vertices of $G(f,K)$ are the $K$-rational preperiodic points for $f$, and there is a directed edge from $\alpha$ to $\beta$ if and only if $f(\alpha) = \beta$.

Northcott proved in \cite[Thm. 3]{northcott:1950} that $\PrePer(f,K)$ is a finite set. Based on the analogy between preperiodic points for rational maps and torsion points on elliptic curves, Morton and Silverman have conjectured a dynamical analogue of the strong uniform boundedness conjecture (now Merel's theorem \cite{merel:1996}) for elliptic curves:

\begin{conj}[{\cite[p. 100]{morton/silverman:1994}}]\label{conj:ubc}
Fix $n \ge 1$ and $d \ge 2$. There is a constant $C(n,d)$ such that for any number field $K$ of absolute degree $n$, and for any rational map $f \in K(z)$ of degree $d$,
	\[
		\#\PrePer(f,K) \le C(n,d).
	\]
\end{conj}

It is currently unknown whether such a constant exists for any pair of integers $(n,d)$, even if one restricts to polynomial maps. Even the simplest polynomial case $(n,d) = (1,2)$---that is, quadratic polynomials over $\bbQ$---has not yet been proven unconditionally, though Looper \cite{looper} has recently given a proof assuming a generalization of the $abc$-conjecture over $\bbQ$.

The difficulty in proving Conjecture~\ref{conj:ubc} in this case is bounding the possible {\it periods} of rational periodic points---see \cite[Cor. 1.11]{doyle/poonen:2020} for an explicit statement to this effect. It is shown in \cite{walde/russo:1994} that for each period $n \in \{1,2,3\}$, there are infinitely many quadratic polynomials (up to an appropriate notion of equivalence) with a rational point of period $n$. On the other hand, there are no quadratic polynomials with rational points of period $4$ (\cite[Thm. 4]{morton:1998}), period $5$ (\cite[Thm. 1]{flynn/poonen/schaefer:1997}), or---assuming standard conjectures on $L$-series for the Jacobian of a certain curve of genus $4$---period $6$ (\cite[Thm. 7]{stoll:2008}). It was conjectured in \cite{flynn/poonen/schaefer:1997} that no quadratic polynomial over $\bbQ$ could have a rational point of period greater than 3, and Poonen has shown that this conjecture would imply uniform boundedness for quadratic polynomials over $\bbQ$, analogous to Mazur's theorem \cite{mazur:1977} for rational torsion points on elliptic curves.

\begin{thm}[{\cite[Cor. 1]{poonen:1998}}]\label{thm:poonen}
Let $f \in \bbQ[z]$ be a quadratic polynomial. If $f$ does not admit rational points of period greater than 3, then $G(f,\bbQ)$ is isomorphic to one of the following twelve directed graphs, which appear in Appendix~\ref{app:all_data}:
	\begin{center}
	\rm0, 2(1), 3(1,1), 3(2), 4(1,1), 4(2), 5(1,1)a, 6(1,1), 6(2), 6(3), 8(2,1,1), 8(3).
	\end{center}
In particular, $\#\PrePer(f,\bbQ) \le 9$.
\end{thm}

\begin{rem}\label{rem:infinity}
The bound from Theorem~\ref{thm:poonen} is 9, while the largest graph appearing in the classification has eight vertices. This is due to the fact that, following the convention of \cite{poonen:1998,doyle/faber/krumm:2014,doyle:2018quad}, we omit the fixed point at infinity when describing the preperiodic graph for a polynomial map.
\end{rem}

A reasonable next step in studying preperiodic points for quadratic polynomials is to give a classification like Poonen's, but over quadratic extensions of $\bbQ$. Conjecture~\ref{conj:ubc} suggests that there should be only finitely many isomorphism classes of graphs $G(f,K)$ with $K$ a quadratic field and $f \in K[z]$ quadratic, so there are two natural directions to pursue:
\renewcommand{\labelenumi}{(\arabic{enumi})}
	\begin{enumerate}
	\item Classify those graphs $G$ that may be realized as $G(f,K)$ for some quadratic field $K$ and some quadratic $f \in K[z]$.
	\item Fix a collection of quadratic fields $K$, and for each classify those graphs $G$ that may be realized as $G(f,K)$ for some quadratic $f \in K[z]$.
	\end{enumerate}
\renewcommand{\labelenumi}{(\Alph{enumi})}
A classification as in (1) would be a dynamical analogue of the corresponding result by Kamienny \cite{kamienny:1992} and Kenku-Momose \cite{kenku/momose:1988} for quadratic torsion points on elliptic curves; the articles \cite{doyle/faber/krumm:2014,doyle:2018quad} (with \cite{doyle/krumm/wetherell} in preparation) give progress in this direction. In the current article, we consider direction (2), giving a conditional classification like Theorem~\ref{thm:poonen} for the quadratic cyclotomic fields $\bbQ(i)$ and $\bbQ(\omega)$, where $i = \sqrt{-1}$ and $\omega = (-1 + \sqrt{-3})/2$ are primitive fourth and third roots of unity, respectively. We now state our main result, which should be viewed as a conditional analogue of classification results due to Najman \cite{najman:2011, najman:2010} for torsion on elliptic curves defined over these two quadratic fields.

\begin{thm}\label{thm:main}\mbox{}
	\begin{enumerate}
		\item Let $K = \bbQ(i)$, and let $f \in K[z]$ be a quadratic polynomial that does not admit $K$-rational points of period greater than 5. Then $G(f,K)$ is isomorphic to one of the following fourteen graphs:
			\begin{center}
				\rm 0, 3(2), 4(1,1), 4(2), 5(1,1)a/b, 5(2)a, 6(1,1), 6(2), 6(2,1), 6(3),
				8(2,1,1), 8(3), 10(2,1,1)a.
			\end{center}
		\item Let $K = \bbQ(\omega)$, and let $f \in K[z]$ be a quadratic polynomial that does not admit $K$-rational points of period greater than 5. Then $G(f_c,K)$ is isomorphic to one of the following thirteen graphs:
			\begin{center}
				\rm 0, 3(2), 4(1), 4(1,1), 4(2), 5(1,1)a, 6(1,1), 6(2), 6(3), 7(2,1,1)a, 8(2)a, 8(2,1,1), 8(3).
			\end{center}
	\end{enumerate}
\end{thm}

The proofs of all of the results referred to above for torsion points on elliptic curves relied heavily on the use of modular curves, which parametrize elliptic curves (up to isomorphism) together with marked points of a given order. Perhaps unsurprisingly, much of the corresponding work on preperiodic points for quadratic polynomials has relied on {\it dynamical modular curves,} which parametrize quadratic polynomials (up to dynamical equivalence) together with marked preperiodic points. In particular, Theorem~\ref{thm:poonen} required finding the full set of rational points on several dynamical modular curves, just as the articles \cite{doyle/faber/krumm:2014, doyle:2018quad, doyle/krumm/wetherell} involve finding {\it quadratic} points on such curves; this is the strategy we employ in the current article.

We give a brief overview of dynamical modular curves in \textsection\ref{sec:dmc}, and we collect in \textsection\ref{sec:quad} several results that will be useful for determining the set of quadratic points on such curves. Section~\ref{sec:genus_gonality} gives some results on dynamical modular curves of low genus, which are later used to prove our main theorem. Finally, \textsection\ref{sec:specific} contains the proof of Theorem~\ref{thm:main}, which has been split into two statements, Propositions~\ref{prop:main_i} and \ref{prop:main_omega}. We also include in an appendix the complete determination of the set of $\bbQ(i)$-rational points on the dynamical modular curve $X_0(5)$, which parametrizes quadratic polynomials together with a marked periodic cycle of length $5$. The calculation involves a slight variant of the usual Chabauty-Coleman method.

\subsection*{A remark on computations}
Nearly all of the required computations were carried out using Magma \cite{magma_cyclotomicquad}, though a few were done in Sage \cite{sage_cyclotomicquad}. We have included, as ancillary files to this article's arXiv submission \cite{doyle:cyclotomic}, three files containing code and output. The first ({\tt main.txt}) contains the calculations for the main body of the paper, while the second ({\tt Chabauty.txt}) and third ({\tt KummerSurface.txt}) include the computations required for Appendix~\ref{app:X0(5)}.

\subsection*{Acknowledgments} This article began as part of my dissertation at the University of Georgia, though several additions and improvements have been made since that time. I thank my advisor, Bob Rumely, for many insightful conversations and for his guidance during my time at Georgia. I thank Pete Clark for his help with some of the background on algebraic curves and Bjorn Poonen for helpful discussions. In addition to providing several helpful comments, I would like to thank the anonymous referee for the motivation to prove Theorem~\ref{thm:5cycle_cyclotomic} over $\bbQ(i)$ (the original version only included the statement for $\bbQ(\omega)$), and I am indebted to Joseph Wetherell for introducing me to the method used in \textsection \ref{sec:ChabCalc} to do so.

\section{Dynamical modular curves for quadratic polynomial maps}\label{sec:dmc}

In this section, $K$ will be a number field. Given a finite directed graph $G$, we describe a \emph{dynamical modular curve} whose $K$-rational points parametrize quadratic maps $f \in K[z]$, up to equivalence, together with a collection of marked points that ``generate" a subgraph of $G(f,K)$ isomorphic to $G$. Throughout this article, we will use the notation $X_1(\cdot)$, $Y_1(\cdot)$, and $U_1(\cdot)$ exclusively to represent various dynamical modular curves. When we need to refer to a \emph{classical} modular curve, which parametrizes elliptic curves together with certain level structure, we will heed the advice of \cite[p. 163]{silverman:2007} and write $\Xell_1(\cdot)$ and $\Yell_1(\cdot)$ to avoid confusion.

We first describe what we mean by dynamical equivalence: We say that two polynomial maps $f,g \in K[z]$ are {\bf linearly conjugate} if there exists a polynomial $\ell(z) = az + b$, with $a,b \in K$ and $a \ne 0$, such that $g = f^\ell := \ell^{-1} \circ f \circ \ell$. Linear conjugation is the appropriate notion of equivalence dynamically since conjugation commutes with iteration. In particular, $\ell$ induces a graph isomorphism $G(g,K) \overset{\sim}{\longrightarrow} G(f,K)$. It is well known that every quadratic polynomial over a field $K$ of characteristic 0 is linearly conjugate to a unique polynomial of the form
	\[
		f_c(z) := z^2 + c
	\]
with $c \in K$, so it suffices to restrict our attention to maps of this form.

In this paper, we give only an informal description of these dynamical modular curves. A more formal treatment appears in \cite{doyle:2019}.

\subsection{Dynatomic curves}\label{sub:per_dyn}

Let $N$ be any positive integer. If $x$ is a point of period $N$ for $f_c$, then we have $f_c^N(x) - x = 0$. However, this equation is also satisfied if $x$ has period equal to a proper divisor of $N$. One therefore defines the \textbf{$N$th dynatomic polynomial} to be
	\[
		\Phi_N(x,c) := \prod_{n \mid N} \left(f_c^n(x) - x\right)^{\mu(N/n)} \in \bbZ[x,c],
	\]
where $\mu$ is the M\"{o}bius function. The dynatomic polynomials provide a natural factorization
	\begin{equation}\label{eq:Ncycle}
		f_c^N(x) - x = \prod_{n \mid N} \Phi_n(x,c)
	\end{equation}
for all $N \in \bbN$---see \cite[p. 571]{morton/vivaldi:1995}. If $(x,c) \in K^2$ satisfies $\Phi_N(x,c) = 0$, we say that $x$ has \textbf{formal period} $N$ for $f_c$. Every point of exact period $N$ has formal period $N$, but in some cases a point of formal period $N$ may have exact period $n$ a proper divisor of $N$. The fact that $\Phi_N(x,c)$ is a polynomial is shown in \cite[Thm. 4.5]{silverman:2007}. If we define\label{eq:r(N)}
	\begin{align*}
	d(N) &:= \deg_x \Phi_N(x,c) = \sum_{n \mid N} \mu(N/n) 2^n,\\
	r(N) &:= \frac{d(N)}{N},
	\end{align*}
then $d(N)$ (resp., $r(N)$) denotes the number of points (resp., cycles) of period $N$ for a generic quadratic polynomial map.

Since $\Phi_N(x,c)$ has coefficients in $\bbZ$, the equation $\Phi_N(x,c) = 0$ defines an affine plane curve $Y_1(N)$ over $K$, and this curve was shown to be irreducible over $\bbC$ by Bousch \cite[\textsection 3, Thm. 1]{bousch:1992}. We define $U_1(N)$ to be the Zariski open subset of $Y_1(N)$ on which $\Phi_n(x,c) \ne 0$ for each proper divisor $n$ of $N$. In other words, $(x,c)$ lies on $Y_1(N)$ (resp., $U_1(N)$) if and only if $x$ has formal (resp., exact) period $N$ for $f_c$. We denote by $X_1(N)$ the normalization of the projective closure of $Y_1(N)$.

Given a collection of pairwise distinct positive integers $N_1,\ldots,N_m$, we let $Y_1(N_1,\ldots,N_m)$ be the curve given as the subscheme of $\bbA^{m+1}$ defined by
\[
	\Phi_{N_1}(x_1,c) = \cdots = \Phi_{N_m}(x_m,c) = 0,
\]
and we let $X_1(N_1,\ldots,N_m)$ be the normalization of the projective closure of $Y_1(N_1,\ldots,N_m)$.

\subsection{Generalized dynatomic curves}
More generally, suppose $\alpha$ has preperiodic portrait $(M,N)$ for $f_c$ for some $M \ge 0$ and $N \ge 1$. In this case, we have $f_c^{M+N}(\alpha) - f_c^M(\alpha) = 0$; however, this equation is satisfied whenever $\alpha$ has portrait $(m,n)$ for some $0 \le m \le M$ and $n \mid N$. Therefore, for a pair of positive integers $M,N$, we define the \textbf{generalized dynatomic polynomial}
	\[ \Phi_{M,N}(x,c) := \frac{\Phi_N(f_c^M(x), c)}{\Phi_N(f_c^{M-1}(x), c)} \in \bbZ[x,c], \]
and we extend this definition to $M = 0$ by setting $\Phi_{0,N} := \Phi_N$. That $\Phi_{M,N}$ is a polynomial is proven in \cite[Thm. 1]{hutz:2015}. The generalized dynatomic polynomials give a natural factorization
	\[
		f_c^{M+N}(x) - f_c^M(x) = \prod_{m=0}^M\prod_{n \mid N} \Phi_{m,n}(x,c)
	\]
for all $M \ge 0$ and $N \ge 1$. If $\Phi_{M,N}(\alpha,c) = 0$, we say that $\alpha$ has \textbf{formal (preperiodic) portrait} $(M,N)$ for $f_c$. Just as in the periodic case, every point of exact portrait $(M,N)$ has formal portrait $(M,N)$, but the converse is not true in general. 

Let\footnote{Note the subtle distinction between $Y_1(M,N)$, defined at the end of \textsection \ref{sub:per_dyn}, and $Y_1((M,N))$.} $Y_1((M,N))$ be the affine plane curve defined by $\Phi_{M,N}(x,t) = 0$. That these curves are irreducible over $\bbC$ follows from the work of Bousch \cite[p. 67]{bousch:1992}. We define $U_1((M,N))$ to be the Zariski open subset of $Y_1((M,N))$ given by
\begin{equation}\label{eq:phiMNconditions}
	\Phi_{m,n}(x,c) \ne 0 \text{\ for all $m < M$ and $n < N$ (with $n \mid N$)},
\end{equation}
and we denote by $X_1((M,N))$ the normalization of the projective closure of $Y_1((M,N))$.
Note that a point $(\alpha,c)$ lies on $Y_1((M,N))$ (resp., $U_1((M,N))$) if and only if $\alpha$ has formal portrait (resp., exact portrait) $(M,N)$ for $f_c$.

\subsection{Admissible graphs}\label{sub:admissible}
Given a number field $K$ and a parameter $c \in K$, the graph $G(f_c,K)$ necessarily has a great deal of structure and symmetry dictated by the dynamics of quadratic polynomial maps. For this reason, we restrict our attention to finite directed graphs $G$ that possess this additional structure.

\begin{defn}\label{defn:admissible}
A finite directed graph $G$ is \textbf{admissible} if it has the following two properties:

\renewcommand{\labelenumi}{(\alph{enumi})}
\begin{enumerate}
\item Every vertex of $G$ has out-degree 1 and in-degree either 0 or 2.
\item For each $N \ge 2$, $G$ contains at most $r(N)$ $N$-cycles. (See the definition of $r(N)$ on page~\pageref{eq:r(N)}.)
\end{enumerate}

We say that $G$ is \textbf{strongly admissible} if it satisfies the following additional condition:
\begin{enumerate}
\setcounter{enumi}{2}
\item If $G$ contains a fixed point (i.e., a vertex with a self-loop), then $G$ contains exactly two such vertices.
\end{enumerate}
\renewcommand{\labelenumi}{(\Alph{enumi})}
\end{defn}

Strong admissibility is a property shared by nearly all preperiodic graphs $G(f_c,K)$. Condition (a) can only fail if there is a vertex of in-degree $1$, which happens if and only if the critical point $0$ is preperiodic. Condition (c) can only fail if $f_c$ has exactly one fixed point, and it is well known that only $c = 1/4$ has this property. To summarize, we have the following:

\begin{lem}[{\cite[Lem. 2.4 \& Cor. 2.6]{doyle:2019}}]\label{lem:admissible}
Let $K$ be a number field, and let $c \in K$. The graph $G(f_c,K)$ is \textit{admissible} if and only if $0 \notin \PrePer(f_c,K)$ and is \textit{strongly admissible} if and only if $0 \notin \PrePer(f_c,K)$ and $c \ne 1/4$. In particular, the set of parameters $c \in K$ for which $G(f_c,K)$ is not strongly admissible is finite.
\end{lem}

Given an admissible graph $G$, we define the \textbf{cycle structure} of $G$ to be the nondecreasing list of lengths of disjoint cycles occurring in $G$. We will say that $G$ \emph{contains} the cycle structure $\tau = (N_1,\ldots,N_m)$ if $\tau$ is a subsequence of the cycle structure of $G$; that is, if $G$ has an admissible subgraph with cycle structure $\tau$.

Before discussing the dynamical modular curves associated to admissible graphs, we require one more definition.

\begin{defn}\label{defn:generate}
Let $G$ be an admissible graph, and let $\{P_1,\ldots,P_n\}$ be a set of vertices of $G$. Let $H$ be the smallest admissible subgraph of $G$ containing all of the vertices $P_1,\ldots,P_n$. We say that $\{P_1,\ldots,P_n\}$ is a \textbf{generating set} for $H$. If any other generating set for $G$ contains at least $n$ vertices, then we call $\{P_1,\ldots,P_n\}$ a \textbf{minimal generating set} for $G$.
\end{defn}

We now describe dynamical modular curves $X_1(G)$ associated to admissible graphs $G$, generalizing those curves $X_1(N)$ and $X_1((M,N))$ defined above. The curves $X_1(G)$ are formally defined in \cite{doyle:2019}, where it is shown that $X_1(G)$ is always an irreducible curve in characteristic $0$. For the purposes of this article, however, we will be content to describe $X_1(G)$ as a curve over $\bbC$ as follows: Let $\{P_1,\ldots,P_n\}$ be a minimal generating set for $G$. For a subfield $L \subseteq \bbC$, define\footnote{For simplicity, the definition of $U_1(G)$ given here is slightly more restrictive than our definition of $U_1(G)$ in \cite{doyle:2019}. The difference is that in \cite{doyle:2019}, there were finitely many additional points $(\alpha_1,\ldots,\alpha_n,c)$ on $U_1(G)$ for which $0$ is in the orbit of $\alpha_i$ under $f_c$ for some $i \in \{1,\ldots,n\}$, which forces inadmissibility of the associated preperiodic graph by Lemma~\ref{lem:admissible}.} $U_1(G)(L)$ to be the set of all tuples $(\alpha_1,\ldots,\alpha_n,c) \in \bbA^{n+1}(L)$ such that $\{\alpha_1,\ldots,\alpha_n\}$ generates a subgraph of $\PrePer(f_c,L)$ isomorphic to $G$ via an identification $P_i \longmapsto \alpha_i$. Let $Y_1(G)$ be the Zariski closure in $\bbA^{n+1}$ of the set $U_1(G)(\bbC)$, and let $X_1(G)$ be the normalization of the projective closure of $Y_1(G)$. Note that if $G$ is generated by a single vertex of portrait $(M,N)$, then $X_1(G) = X_1((M,N))$, and similarly for $Y_1(G)$ and $U_1(G)$.

The assignment $G \longmapsto X_1(G)$ is (contravariant) functorial: If $G$ and $H$ are admissible graphs with $H \subseteq G$, there is a nonconstant map $X_1(G) \longto X_1(H)$ that commutes with projection onto the $c$-line; see \cite[Prop. 3.3]{doyle:2019}.

\begin{figure}
\centering
\begin{overpic}[scale=.5]{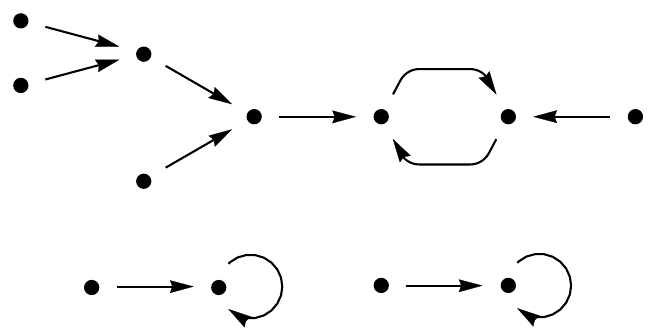}
	\put(44,0){$\alpha$}
	\put(-7,52){$\beta$}
	\put(110,0){$\alpha'$}
\end{overpic}
\caption{An admissible graph $G$}
\label{fig:U1_ex}
\end{figure}

We end this section with an example: The graph $G$ in Figure~\ref{fig:U1_ex} is strongly admissible and is minimally generated by the vertices $\alpha$, $\alpha'$, and $\beta$. Therefore, for a number field $K$ we have
	\begin{align*}
		U_1(G)(K) = \{(\alpha,\alpha',\beta,c) \in \bbA^4(K) : & \text{ $\alpha$ and $\alpha'$ are distinct fixed points for $f_c$;}\\
			& \text{ $\beta$ has portrait $(3,2)$ for $f_c$; and}\\
			& \text{ $0$ is not in the orbit of $\alpha$, $\alpha'$, or $\beta$ under $f_c$}\}.
	\end{align*}

\section{Quadratic points on algebraic curves}\label{sec:quad}

Let $K$ be a number field, and let $X$ be an algebraic curve defined over $K$. We say that $P \in X(\Kbar)$ is \textbf{quadratic over $K$} if the field of definition of $P$, denoted $K(P)$, is a quadratic extension of $K$. We will mostly be working in the situation that $K = \bbQ$, in which case we will simply say that $P$ is {\bf quadratic}.
If $X$ is an affine curve, then the \textbf{genus} of $X$, denoted $g(X)$, will be understood to be the geometric genus of $X$; i.e., the genus of the nonsingular projective curve birational to $X$. A {\bf Weierstrass point} on $X$ is a point $P \in X$ for which there exists a rational map of degree at most $g(X)$ vanishing only at $P$. (Equivalently, a Weierstrass point is one for which there exists a nonconstant rational map of degree at most $g(X)$ which is regular away from $P$).

\subsection{Hyperelliptic curves}\label{sec:hyp}
Several of the curves we consider in this paper are hyperelliptic. Recall that a hyperelliptic curve of genus $g$ defined over $K$ has an affine model of the form $y^2 = f(x)$ for some polynomial $f(x) \in K[x]$ of degree $2g + 1$ or $2g + 2$ with no repeated roots. We denote by $\iota$ the hyperelliptic involution on $X$:
	\[
		\iota(x,y) = (x,-y),
	\]
and we say that $\iota P$ is the {\bf hyperelliptic conjugate} of $P$. If $\deg f$ is odd, then $X$ has a single point at infinity, which is necessarily $K$-rational; if $\deg f$ is even, then $X$ has two points at infinity, and they are $K$-rational if and only if the leading coefficient of $f$ is a square in $K$. This can be seen by covering $X$ by two affine patches: The first is given by the equation $y^2 = f(x)$, and the second is given by $v^2 = u^{2g + 2}f(1/u)$, with the identification $x = 1/u$, $y = v/u^{g+1}$. If $\deg f$ is even, and if $c$ is the leading coefficient of $f$, then we take $\infty^\pm$ to be the two points on $X$ corresponding to $(u,v) = (0,\pm \sqrt{c})$. If $\deg f$ is odd, then we take $\infty^+ = \infty^- = \infty$ to be the unique point at infinity, given by $(u,v) = (0,0)$. In either case, we have $\iota \infty^\pm = \infty^\mp$.

Weierstrass points on hyperelliptic curves are simple to describe: they are precisely the ramification points for the double cover of $\bbP^1$ given by $(x,y) \longmapsto x$; equivalently, they are the fixed points of the hyperelliptic involution. More concretely, $P$ is a Weierstrass point on $X$ if and only if $P = (x,0)$ or $\deg f$ is odd and $P$ is the point at infinity. In particular, every hyperelliptic curve of genus $g$ has $2g+2$ Weierstrass points.

\subsubsection{Curves of genus $2$}
We now focus on curves of genus $2$; much of what follows may be found in \cite{cassels/flynn:1996}. Let $X$ be a curve of genus $2$ defined over a number field $K$. Since every genus $2$ curve $X$ is hyperelliptic, $X$ has an affine model of the form $y^2 = f(x)$ with $f(x) \in K[x]$ of degree $d \in \{5,6\}$ having no repeated roots. Let $J$ be the Jacobian of $X$. If we assume that $X$ has a $K$-rational point (this is guaranteed if $d = 5$), then we may identify the Mordell-Weil group $J(K)$ with the group of $K$-rational degree $0$ divisors of $X$ modulo linear equivalence (see \cite[p. 39]{cassels/flynn:1996} and \cite[p. 168]{milne:1986}, for example). For $n \in \bbN$, we set
	\[
		J(K)[n] := \{\calP \in J(K) : n\calP = \calO\},
	\]
and we let $J(K)_{\tors} := \bigcup_{n \in \bbN} J(K)[n]$ be the full torsion subgroup of $J(K)$.

The divisor $\infty^+ + \infty^-$ is a $K$-rational divisor on $X$, and the divisor class $\calK$ containing $\infty^+ + \infty^-$ is the canonical divisor class. By the Riemann-Roch theorem, every degree $2$ divisor class $\calD$ contains an effective divisor, and this effective divisor is unique if and only if $\calD \ne \calK$. The effective divisors in the canonical class $\calK$ are precisely those of the form $P + \iota P$. We may therefore represent every nontrivial element of $J(K)$ \emph{uniquely} by a divisor of the form $P + Q - \infty^+ - \infty^-$, up to reordering of $P$ and $Q$, where $P + Q$ is a $K$-rational divisor on $X$ (either $P$ and $Q$ are both $K$-rational points on $X$ or $P$ and $Q$ are Galois conjugate quadratic points on $X$). We therefore represent points of $J(K)$ as unordered pairs $\{P,Q\}$, with the identification
	\[ \{P,Q\} = \left[ P + Q - \infty^+ - \infty^- \right], \]
where $[D]$ denotes the divisor class of the divisor $D$. Note that $\{P,Q\} = \calO$ if and only if $[P + Q] = \calK$; that is, if and only if $P$ and $Q$ are hyperelliptic conjugates. It follows that $-\{P,Q\} = \{\iota P, \iota Q\}$, since
	\[
		\{P,Q\} + \{\iota P, \iota Q\}
			= \{P, \iota P\} + \{Q, \iota Q\} = \calO.
	\]
Note that we have a morphism $X \longto J$ obtained by mapping $P \longmapsto \{P,P\}$; unlike the standard Albanese map $P \longmapsto [P - P_0]$ (for a fixed base point $P_0$), this map is not an embedding, since every Weierstrass point maps to $\calO$.

The following statement regarding 2-torsion on genus 2 curves is well known:

\begin{lem}\label{lem:2torsion}
Let $X$ be a genus $2$ curve defined over a number field $K$, let $J$ be its Jacobian, and let $\{P_1,\ldots,P_6\}$ be the set of Weierstrass points on $X$. Then the set of points on $J$ of exact order $2$ is given by
	\[
		J(\Kbar)[2] \setminus \{\calO\} = \{\{P_i,P_j\} :  i \ne j\}.
	\]
\end{lem}

\begin{proof}
The Jacobian $J$ has 15 points of order 2, and that is precisely the number of unordered pairs $\{P_i,P_j\}$ with $i \ne j$, so it suffices to show that each $\{P_i,P_j\}$ is a nonzero 2-torsion point. That $\{P_i,P_j\} \ne \calO$ follows from the fact that $\iota P_i = P_i \ne P_j$, and $\{P_i,P_j\}$ has order $2$ since
	\[
		-\{P_i,P_j\} = \{\iota P_i, \iota P_j\} = \{P_i, P_j\}.
	\]
\end{proof}

\subsection{Points on curves and their Jacobians after base extension}

\subsubsection{A sufficient condition for $X(L) = X(K)$}
Let $A$ be an abelian variety over a field $K$. If $L$ is an extension of $K$, then certainly $A(K) \subseteq A(L)$; in this section we give a sufficient condition for equality to hold. We then give a consequence for rational points on a curve $X$ over a number field $K$ upon base change to a finite Galois extension $L/K$.

If $G$ is a finitely generated abelian group and $H \subseteq G$ is a subgroup, then the \textbf{saturation} of $H$ in $G$ is the largest subgroup $H' \subseteq G$ containing $H$ such that $[H':H] < \infty$. We will say that $H$ is \textbf{saturated} in $G$ if the saturation of $H$ in $G$ is $H$ itself.

\begin{prop}\label{prop:saturation}
Let $K$ be a field, and let $A$ be an abelian variety defined over $K$ such that $A(K)$ is finitely generated. Let $L$ be a finite Galois extension of $K$ of degree $n := [L:K]$. Suppose that $A(L)_{\tors} = A(K)_{\tors}$ and $A(K)[n] = 0$. Then $A(K)$ is saturated in $A(L)$.
\end{prop}

\begin{rem}
If $K$ is finitely generated over its prime subfield, for example, then N\'{e}ron's generalization \cite{neron:1952} of the Mordell-Weil theorem states that $A(K)$ is necessarily finitely generated.
\end{rem}

\begin{proof}[Proof of Proposition~\ref{prop:saturation}]
Let $\calP \in A(L)$ lie in the saturation of $A(K)$; we claim that $\calP \in A(K)$. Let $\sigma \in G := \Gal(L/K)$ be arbitrary. Since $\calP$ lies in the saturation of $A(K)$, there exists some $\ell \in \bbZ$ for which $\ell \calP \in A(K)$, and therefore $\sigma(\ell \calP) = \ell \calP$. Writing this as $\ell (\sigma \calP - \calP) = 0$ shows that $\sigma \calP - \calP$ must be a torsion element of $A(L)$. Since $A(L)_{\tors} = A(K)_{\tors}$, $\sigma \calP - \calP$ must in fact be $K$-rational, so $\tau(\sigma \calP - \calP) = \sigma \calP - \calP$ for all $\tau \in G$. Thus
	\begin{equation}\label{eq:GaloisEq}
		n(\sigma \calP - \calP)
			= \sum_{\tau \in G} (\sigma \calP - \calP)
			= \sum_{\tau \in G} \tau\left(\sigma \calP - \calP\right)
			= \sum_{\tau \in G} \tau\sigma \calP - \sum_{\tau \in G} \tau \calP
			= 0	
	\end{equation}
Since we assumed that $A(K)[n] = 0$, it follows that $\sigma \calP = \calP$. Since this holds for all $\sigma \in G$, we have $\calP \in A(K)$.
\end{proof}

\begin{cor}\label{cor:saturation_samerank}
Let $K$ be a field, and let $A$ be an abelian variety defined over $K$ such that $A(K)$ is finitely generated. Let $L$ be a finite Galois extension of $K$ of degree $n := [L:K]$. Suppose that $\rk A(L) = \rk A(K)$, $A(L)_{\tors} = A(K)_{\tors}$, and $A(K)[n] = 0$. Then $A(L) = A(K)$.
\end{cor}

\begin{proof}
Since $A(K)$ and $A(L)$ have the same rank, the index $[A(L):A(K)]$ is finite. Therefore $A(L)$ is the saturation of $A(K)$ in $A(L)$, so $A(L) = A(K)$ by Proposition~\ref{prop:saturation}.
\end{proof}

\begin{ex}
It is clear that the conditions $\rk A(L) = \rk A(K)$ and $A(L)_{\tors} = A(K)_{\tors}$ are necessary for the conclusion of Corollary~\ref{cor:saturation_samerank}. We now give an example to show that we cannot, in general, omit the restriction on the $n$-torsion. Let $K = \bbQ$, let $A$ be the elliptic curve defined by $y^2 + xy = x^3 - x$, and let $L = \bbQ(\sqrt{5})$. This curve appears as 65A1 in Cremona's table \cite{cremona:1997}, where one finds that $\rk A(\bbQ) = 1$ and $A(\bbQ)_{\tors} \cong \bbZ/2\bbZ$; in particular, the $n$-torsion condition fails in this case. A computation in \cite[{\tt main.txt}]{doyle:cyclotomic} shows that $\rk A(L) = \rk A(\bbQ) = 1$ and $A(L)_{\tors} = A(\bbQ)_{\tors} \cong \bbZ/2\bbZ$. However, we have $\left(\frac{1 + \sqrt{5}}{2}, 1 \right) \in A(L) \setminus A(\bbQ)$.
\end{ex}

We will ultimately apply Corollary~\ref{cor:saturation_samerank} to Jacobian varieties of curves. If $J$ is the Jacobian of a curve $X$, then knowing that $J(L) = J(K)$ essentially determines $X(L)$ from $X(K)$. We make this precise with the following proposition:

\begin{prop}\label{prop:same_jac}
Let $X$ be a curve of genus $g \ge 2$ defined over a number field $K$, and let $J$ be the Jacobian of $X$. Let $L/K$ be a Galois extension, and suppose that $J(L) = J(K)$. Then one of the following must be true:
	\begin{enumerate}
	\item $X(L) = X(K)$, or
	\item $X$ is hyperelliptic, $X(K) = \emptyset$, and $X(L)$ consists entirely of Weierstrass points.
	\end{enumerate}
\end{prop}

\begin{proof}
Assume $X(L) \supsetneq X(K)$, and let $P \in X(L) \setminus X(K)$.

First, suppose for contradiction that there is a point $Q \in X(K)$. Since $[P - Q] \in J(L) = J(K)$, we have
	\[
		[\sigma P - Q] = [P - Q]^\sigma = [P - Q]
	\]
for all $\sigma \in \Gal(L/K)$. This implies that $[\sigma P - P]$ is the trivial divisor class for all $\sigma \in \Gal(L/K)$; since $g > 0$, it must be that $\sigma P = P$ for all $\sigma \in \Gal(L/K)$, contradicting our assumption that $P \notin X(K)$. Therefore, $X(K) = \emptyset$.

Now, choose any $\sigma \in \Gal(L/K)$ for which $\sigma P \ne P$. Since $[\sigma P - P] \in J(L) = J(K)$, we have
	\[
		[P - \sigma^{-1}P] = [\sigma P - P]^{\sigma^{-1}} = [\sigma P - P],
	\]
hence $[2P - \sigma P - \sigma^{-1} P]$ is the trivial class. By assumption, neither $\sigma P$ nor $\sigma^{-1} P$ is equal to $P$, so there is a degree 2 rational map $f$ on $X$ that vanishes only at $P$. The fact that $\deg f = 2$ implies that $X$ is hyperelliptic, and the fact that $\deg f \le g$ implies that $P$ is a Weierstrass point. 
\end{proof}

\begin{ex}
We give an example to show that the situation described in part (B) of Proposition~\ref{prop:same_jac} does occur. Let $X$ be the genus 2 curve given by
	\[
		y^2 = (x^2 + 1)(2x^4 + x^3 + 2x^2 + 2x + 2),
	\]
and let $L = \bbQ(i)$. Then $J(L) = J(\bbQ) = \{\calO, \{P^+, P^-\}\}$, where $P^{\pm} = (\pm i, 0)$ are the two $L$-rational Weierstrass points on $X$. In this case, we have $X(\bbQ) = \emptyset$ and $X(L) = \{P^+,P^-\}$.
\end{ex}

\subsubsection{Mordell-Weil ranks and quadratic twists}

Finally, we record a lemma which we use repeatedly for calculating ranks of Jacobian varieties over quadratic fields. Nothing in this section is new, but we include full details for completeness.

Let $K$ be a number field, let $L = K(\sqrt{d})$ be a quadratic extension of $K$, and let $\chi$ be the usual quadratic character defined by
	\begin{align*}
	\chi : \Gal(\Kbar/K) &\longto \{\pm 1\}\\
		\sigma &\longmapsto \sigma(\sqrt{d})/\sqrt{d}.
	\end{align*}
If $A$ is an abelian variety defined over $K$, then the {\it quadratic twist} of $A$ by $d$ is an abelian variety $A^{(d)}$ which is not isomorphic to $A$ over $K$, but for which there is an $L$-isomorphism
	\[
		\phi : A^{(d)} \longto A
	\]
such that $\phi^\sigma = \chi(\sigma) \cdot \phi$ for all $\sigma \in \Gal(\Kbar/K)$. See \cite[\textsection X.2]{silverman:2009} for details about twists in the case that $A$ is an elliptic curve and \cite[\textsection C.5]{hindry/silverman:2000} for twists of quasiprojective varieties in general.
We note that if $\tau$ is the generator for $\Gal(L/K)$, then the condition that $\phi^\tau = -\phi$ implies that
	\begin{equation}\label{eq:twist}
	\phi\left(A^{(d)}(K)\right) = \{\calP \in A(L) \mid \tau\calP = -\calP\}.
	\end{equation}

Without loss of generality, we may assume $\phi(0) = 0$. Indeed, setting $\tilde{\phi} = \phi - \phi(0)$ we have
	\begin{align*}
		\tilde{\phi}^\sigma &= \phi^\sigma - \sigma\phi(0)\\
			&= \phi^\sigma - \phi^\sigma(0) \qquad\qquad \text{(since $0 \in A(K)$)}\\
			&= \chi(\sigma)(\phi - \phi(0))\\
			&= \chi(\sigma)\tilde{\phi}.
	\end{align*}
We henceforth assume $\phi(0) = 0$, thus $\phi$ is a group homomorphism from $A^{(d)}(\Kbar) \longto A(\Kbar)$, hence a group isomorphism; see \cite[\textsection II, Cor. 1]{mumford:1970}, for example.

We are interested in the particular case that $A$ is the Jacobian of a (hyper)elliptic curve $X$ defined by $y^2 = f(x)$, in which case $A^{(d)}$ is the Jacobian of the curve $X^{(d)}$ defined by $dy^2 = f(x)$, and the isomorphism $\phi$ is induced by the isomorphism of curves given by
	\begin{align*}
		X^{(d)} &\longto X\\
		(x,y) &\longmapsto \left(x, y\sqrt{d}\right).
	\end{align*}

It is well known that the rank of $A(L)$ is the sum of the ranks of $A(K)$ and $A^{(d)}(K)$; however, we include the proof for completeness, and we thank Pete Clark for pointing us toward this argument.

\begin{lem}\label{lem:rank_twist}
Let $K$ be a number field, and let $L = K(\sqrt{d})$ be a quadratic extension. Let $\chi$ denote the quadratic character associated to $L/K$. Let $A$ be an abelian variety defined over $K$, let $A^{(d)}$ be its quadratic twist by $d$, and let
	\[
		\phi : A^{(d)} \longto A
	\]
be an $L$-isomorphism satisfying $\phi(0) = 0$ and $\phi^\sigma = \chi(\sigma) \cdot \phi$ for all $\sigma \in \Gal(\Kbar/K)$. Consider the group homomorphism
	\begin{align*}
		\psi : A(K) \oplus A^{(d)}(K) &\longto A(L)\\
			(\calP, \calQ) &\longmapsto \calP + \phi(\calQ).
	\end{align*}
Then the kernel ($\ker \psi$) and cokernel ($\coker \psi$) are finite and, moreover, the exponent of $\coker \psi$ is at most $2$. It follows that
	\[
		\rk A(L) = \rk A(K) + \rk A^{(d)}(K).
	\]
\end{lem}

\begin{proof}
Let $\tau$ be the generator for $\Gal(L/K)$. First, suppose $(\calP, \calQ) \in \ker \psi$, so that $\phi(\calQ) = -\calP$. Since $\calQ \in A^{(d)}(K)$, we can apply \eqref{eq:twist} to see that
	\[
		\calP = -\phi(\calQ) = \tau\phi(\calQ) = \tau(-\calP) = -\calP,
	\]
where the last equality follows from the fact that $\calP$ is $K$-rational. Thus $\calP$ lies in the finite subgroup $A(K)[2]$; since $\calQ = \phi^{-1}(-\calP)$ is determined by $\calP$, it follows that $\psi$ has finite kernel.

To show that the cokernel of $\psi$ has exponent $2$, we must show that for all $\calR \in A(L)$, $2\calR$ is in the image of $\psi$. Let $\calP := \calR + \tau\calR$ and $\calQ := \phi^{-1}(\calR - \tau\calR)$. The point $\calP$ is fixed by $\tau$, hence is $K$-rational; on the other hand, $\calR - \tau\calR$ is negated by $\tau$, so $\calQ \in A^{(d)}(K)$ by \eqref{eq:twist}. Finally, we have $\psi(\calP, \calQ) = 2\calR$, so $2\calR$ is in the image of $\psi$, as claimed. Since the Mordell-Weil group $A(L)$ is finitely generated and abelian, so must be $\coker \psi$; thus, the fact that $\coker\psi$ has finite exponent implies that $\coker\psi$ is actually finite.

Finally, since the kernel of $\psi$ is finite and the image of $\psi$ has finite index in $A(L)$, we have
	\begin{align*}
		\rk A(L) &= \rk \psi\left(A(K) \oplus A^{(d)}(K)\right)\\
			&= \rk \left(A(K) \oplus A^{(d)}(K)\right)\\
			&= \rk A(K) + \rk A^{(d)}(K).
	\end{align*}
\end{proof}

\section{Dynamical modular curves of genus at most $2$}\label{sec:genus_gonality}

Because we are concerned with dynamics over quadratic fields, it would be useful to know for which admissible graphs $G$ we should expect $X_1(G)$ to have infinitely many quadratic points. By \cite[Cor. 3]{harris/silverman:1991}, this is equivalent to asking for which admissible graphs $G$ the curve $X_1(G)$ is rational, elliptic, hyperelliptic, or admits a degree 2 morphism to an elliptic curve with positive Mordell-Weil rank over $\bbQ$. In \cite{doyle/krumm/wetherell}, we show that the only such curves $X_1(G)$ are those of genus at most 2; see \cite{ogg:1974, jeon/kim:2004} for similar results for classical modular curves.

In this section, we analyze the torsion subgroups of the Jacobians of the curves $X_1(G)$ of genus at most 2. This analysis, together with Proposition~\ref{prop:same_jac}, will be used in \textsection \ref{sec:specific} to determine the set of $K$-rational points on $X_1(G)$ for certain graphs $G$, with $K = \bbQ(i)$ and $K = \bbQ(\omega)$.

\begin{rem}\label{rem:strongly}
When considering the curves $X_1(G)$ for admissible graphs $G$, we lose no generality by restricting to {\it strongly} admissible graphs. Let $G$ be an admissible graph which is not strongly admissible, which implies that $G$ has a single fixed point. Let $G'$ be the strongly admissible graph obtained from $G$ by adjoining a second fixed point (and, necessarily, its nonperiodic preimage). Then the curve $X_1(G)$ is isomorphic to $X_1(G')$: Indeed, let $K_G/\bbC(c)$ and $K_{G'}/\bbC(c)$ be the function fields of the curves $X_1(G)$ and $X_1(G')$, respectively, in a common algebraic closure of $\bbC(c)$. (Here, we are taking $c$ to be an indeterminate over $\bbC$.) Then $K_{G'}$ is generated over $K_G$ by a root of $\Phi_1(x,c) = x^2 - x + c$; however, $G$ already has one fixed point, hence $K_G$ already contains a root of $\Phi_1(x,c)$, and therefore $K_G$ contains {\it both} roots of $\Phi_1(x,c)$. It follows that $K_{G'} = K_G$, thus $X_1(G') \cong X_1(G)$.
\end{rem}

Since every dynamical modular curve of genus $0$ already has a rational point---hence is isomorphic over $\bbQ$ to $\bbP^1$---we restrict our attention to dynamical modular curves of genus $1$ or $2$. For such curves, we will be interested in determining the torsion subgroups $J_1(G)(K)_{\tors}$ as $K$ ranges over all quadratic extensions $K/\bbQ$.
In order to do so, we require a complete list of all dynamical modular curves of genus $1$ or $2$.

\begin{prop}\label{prop:all_dyn_mod_curves}
Let $G$ be a strongly admissible graph.
	\begin{enumerate}
		\item The curve $X_1(G)$ has genus 0 if and only if $G$ is isomorphic to one of the following:
		\begin{center}
		\rm 4(1,1), 4(2), 6(1,1), 6(2), 6(3), 8(2,1,1).
		\end{center}
		\item The curve $X_1(G)$ has genus 1 if and only if $G$ is isomorphic to one of the following:
		\begin{center}
		\rm 8(1,1)a, 8(1,1)b, 8(2)a, 8(2)b, 10(2,1,1)a, 10(2,1,1)b.
		\end{center}
		\item The curve $X_1(G)$ has genus 2 if and only if $G$ is isomorphic to one of the following:
		\begin{center}
		\rm 8(3), 8(4), 10(3,1,1), 10(3,2).
		\end{center}
	\end{enumerate}
\end{prop}

\begin{proof}
The genera of $X_1(G)$ for those graphs $G$ listed in the statement of the proposition were determined in earlier papers, specifically \cite{walde/russo:1994, morton:1992, morton:1998, poonen:1998}. Thus, we need only show that if $G$ is any strongly admissible graph {\it not} listed, then $g(X_1(G)) > 2$.

The curves $X_1(G)$ naturally form an inverse system, with maps $X_1(G') \longto X_1(G)$ whenever $G \subseteq G'$. Moreover, if $G \subsetneq G'$, the corresponding map of dynamical modular curves has degree at least $2$; these two statements form the content of \cite[Prop. 3.3]{doyle:2019}. Thus, once we have $g(X_1(G)) \ge 2$, it follows from Riemann-Hurwitz that $g(X_1(G')) > 2$ for all $G' \supsetneq G$.

Bousch \cite{bousch:1992} gave an explicit formula for the genera of the curves $X_1(n)$; from that formula, one sees that $g(X_1(n))$ grows on the order of $n2^n$ as $n \to \infty$. The values of $g(X_1(n))$ for small values of $n$ are shown in Table~\ref{tab:genus}, and using Bousch's formula one can verify that $X_1(n)$ has genus greater than $2$ when $n > 4$. It follows that if $X_1(G)$ has genus at most $2$, then either $G \cong {\rm8(4)}$ (the minimal admissible graph with a 4-cycle), in which case $g(X_1(G)) = 2$, or $G$ only contains cycles of length 1, 2, or 3.

\begin{table}
\renewcommand{\arraystretch}{1.5}
\caption{Genera of $X_1(n)$ for small values of $n$}
\label{tab:genus}
	\begin{tabular}{|c|c|c|c|c|c|c|c|c|}
		\hline
		$n$ & 1 & 2 & 3 & 4 & 5 & 6 & 7 & 8\\ \hline
		$g(X_1(n))$ & 0 & 0 & 0 & 2 & 14 & 34 & 124 & 285 \\ \hline
	\end{tabular}
\end{table}

A quadratic polynomial necessarily has at most two fixed points, a single $2$-cycle, and two $3$-cycles. However, Morton \cite{morton:1992} showed that the curve $X_1(3,3)$, which parametrizes maps $f_c$ together with a pair of marked points of period $3$ with {\it disjoint} orbits, has genus $4$. Thus, if $G$ is a strongly admissible graph with $g(X_1(G)) \le 2$, then either $G \cong \rm8(4)$ or the cycle structure of $G$ (defined immediately following Lemma~\ref{lem:admissible}) is one of the following:
\begin{equation}\label{eq:CycleStructures}
	\rm (1,1),\ (2),\ (3),\ (1,1,2),\ (1,1,3),\ (2,3).
\end{equation}

All strongly admissible graphs with eight vertices are listed in the statement of the proposition; in particular, for each such graph $G$ we have $g(X_1(G)) \le 2$. There are only twelve strongly admissible $10$-vertex graphs with one of the cycle structures given above; we list these graphs in Table~\ref{tab:TenVertex}, together with the genera of their dynamical modular curves:\footnote{Given a model for each curve, Magma can easily compute its genus. Models appear in \cite{poonen:1998, doyle/faber/krumm:2014, doyle:2018quad, doyle/krumm/wetherell}.}

\begin{table}
\renewcommand{\arraystretch}{1.5}
\caption{Ten-vertex graphs with cycle structures from \eqref{eq:CycleStructures}}
\label{tab:TenVertex}
	\begin{tabular}{|c|ccccccccc|}\hline
	$G$ & 10(1,1)a/b & 10(2) & 10(3)a/b & 10(2,1,1)a/b & 10(3,1,1) & 10(3,2) & $G_1$ & $G_2$ & $G_3$\\\hline
	$g(X_1(G))$ & 5 & 5 & 9 & 1 & 2 & 2 & 5 & 5 & 5\\\hline
	\end{tabular}
\end{table}

Now suppose $G$ were a graph that did not appear in the statement of the proposition but for which $X_1(G)$ had genus at most $2$. Then $G$ would have to properly contain $\rm 10(2,1,1)a$ or $\rm 10(2,1,1)b$. Moreover, since $\rm 10(3,1,1)$ and $\rm 10(3,2)$ are the smallest admissible graphs containing points of period $1$ and period $3$ (resp., period $2$ and period $3$), and since their curves have genus $2$, $G$ cannot also have a $3$-cycle. Thus, $G$ must have cycle structure $\rm (1,1,2)$. Any admissible graph of cycle structure $\rm (1,1,2)$ that properly contains $\rm 10(2,1,1)a$ or $\rm 10(2,1,1)b$ must contain one of $\rm 12(2,1,1)a/b$, $G_4$, $G_5$, or $G_6$. However, the dynamical modular curve associated to each of these graphs has genus $5$; see \cite{doyle/faber/krumm:2014, doyle:2018quad}. Therefore, the proposition lists all strongly admissible graphs $G$ such that $g(X_1(G)) \le 2$.
\end{proof}

We include in Appendix~\ref{app:models} two key pieces of information for each dynamical modular curve of genus $1$ or $2$: First, we provide an explicit model for each such curve. Second, each point on $X_1(G)$ carries the information of a map $f_c$ together with a collection of preperiodic points; for the models we provide, we include the rational map $X_1(G) \longto \bbP^1$ that maps a point to the corresponding parameter $c$.

\subsection{Curves of genus $1$}\label{sec:gen1}

Each of the dynamical modular curves of genus $1$ has rational points and is therefore isomorphic to an elliptic curve over $\bbQ$. All of these curves have small conductor, so we may refer to Cremona's tables \cite{cremona:1997} to determine the Mordell-Weil groups of these curves over $\bbQ$; in each case, one finds that the rank is $0$, hence $X_1(G)(\bbQ) = X_1(G)(\bbQ)_\tors$. We give in Table~\ref{tab:genus_one_labels} the Cremona labels (found in \cite{poonen:1998}) and rational Mordell-Weil groups for each of the genus $1$ dynamical modular curves.

\begin{table}
\centering
\renewcommand{\arraystretch}{1.5}
\caption{Dynamical modular curves of genus $1$}
\label{tab:genus_one_labels}
\begin{tabular}{|c|c|c|}
\hline
$G$ & Cremona label for $X_1(G)$ & $X_1(G)(\bbQ)$\\ \hline
8(1,1)a & 24A4 & $\bbZ/4\bbZ$ \\  \hline
8(1,1)b & 11A3 & $\bbZ/5\bbZ$ \\ \hline
8(2)a & 40A3 & $\bbZ/4\bbZ$ \\ \hline
8(2)b & 11A3 & $\bbZ/5\bbZ$ \\ \hline
10(2,1,1)a & 17A4 & $\bbZ/4\bbZ$ \\ \hline
10(2,1,1)b & 15A8 & $\bbZ/4\bbZ$ \\ \hline
\end{tabular}
\end{table}

We now determine, for each of the curves $X$ listed in Table~\ref{tab:genus_one_labels}, the torsion subgroup $X(K)_{\tors}$ over all quadratic fields $K$. For the following theorem, we list the genus $1$ dynamical modular curves according to their Cremona labels.

\begin{thm}\label{thm:gen_one_torsion} Let $d$ be a squarefree integer, and let $K = \bbQ(\sqrt{d})$.
	\begin{enumerate}
		\item[] {\rm(11A3)} If $G = {\rm8(1,1)b}$ or $G = {\rm8(2)b}$, then 
		\[	X_1(G)(K)_{\tors} \cong \bbZ/5\bbZ.
		\]
		\item[] {\rm(15A8)} If $G = {\rm10(2,1,1)b}$, then 
		\[	X_1(G)(K)_{\tors} \cong
			\begin{cases}
				\bbZ/2\bbZ \oplus \bbZ/4\bbZ, &\mbox{ if } d = -15;\\
				\bbZ/8\bbZ, &\mbox{ if } d \in \{-3,5\};\\
				\bbZ/4\bbZ, &\mbox{ otherwise}.
			\end{cases}
		\]
		\item[] {\rm(17A4)} If $G = {\rm10(2,1,1)a}$, then 
		\[	X_1(G)(K)_{\tors} \cong
			\begin{cases}
				\bbZ/2\bbZ \oplus \bbZ/4\bbZ, &\mbox{ if } d = 17;\\
				\bbZ/4\bbZ, &\mbox{ otherwise}.
			\end{cases}
		\]
		\item[] {\rm(24A4)} If $G = {\rm8(1,1)a}$, then 
		\[	X_1(G)(K)_{\tors} \cong
			\begin{cases}
				\bbZ/2\bbZ \oplus \bbZ/4\bbZ, &\mbox{ if } d = -3;\\
				\bbZ/8\bbZ, &\mbox{ if } d \in \{-1,3\};\\
				\bbZ/4\bbZ, &\mbox{ otherwise}.
			\end{cases}
		\]		
		\item[] {\rm(40A3)} If $G = {\rm8(2)a}$, then 
		\[	X_1(G)(K)_{\tors} \cong
			\begin{cases}
				\bbZ/2\bbZ \oplus \bbZ/4\bbZ, &\mbox{ if } d = 5;\\
				\bbZ/4\bbZ, &\mbox{ otherwise}.
			\end{cases}
		\]	
	\end{enumerate}
\end{thm}

\begin{proof}
The curves with Cremona labels 11A3, 15A8, and 24A4 are birational to the classical modular curves $\Xell_1(11)$, $\Xell_1(15)$, and $\Xell_1(2,12)$, respectively. Theorem~\ref{thm:gen_one_torsion} was proven for these curves by Rabarison \cite{rabarison:2010}, whose proof relies on the extension of Mazur's theorem to quadratic fields due to Kenku-Momose \cite{kenku/momose:1988} and Kamienny \cite{kamienny:1992}. Our method of proof for the remaining curves---17A4 and 40A3---is more elementary and, though we do not do so here, may be used to give an alternative proof for the curves 11A3, 15A8, and 24A4.

Let $E_{17}$ and $E_{40}$ denote the curves 17A4 and 40A3, respectively, given in \cite{cremona:1997} by the following models:
	\begin{align*}
	E_{17}&:  y^2 + xy + y = x^3 - x^2 - x;\\
	E_{40}&:  y^2 = (x-1)(x^2 + x - 1).
	\end{align*}

The primes 3 and 5 (resp., 3 and 17) are primes of good reduction for $E_{17}$ (resp., $E_{40}$). If $\frakp$ is a prime in $\OK$ lying above the rational prime $p$, then the residue field $k_{\frakp}$ embeds into $\bbF_{p^2}$; by the calculations in \cite[{\tt main.txt}]{doyle:cyclotomic}, we find the following:
\begin{alignat*}{4}
	E_{17}(\bbF_{3^2}) &\cong \bbZ/4\bbZ \oplus \bbZ/4\bbZ & \hspace{10mm} E_{40}(\bbF_{3^2}) &\cong \bbZ/4\bbZ \oplus \bbZ/4\bbZ\\
	E_{17}(\bbF_{5^2}) &\cong \bbZ/2\bbZ \oplus \bbZ/16\bbZ & E_{40}(\bbF_{17^2}) &\cong \bbZ/2\bbZ \oplus \bbZ/160\bbZ.
\end{alignat*}
It follows that both $E_{17}(K)_{\tors}$ and $E_{40}(K)_{\tors}$ embed into $\bbZ/2\bbZ \oplus \bbZ/4\bbZ$. Since both $E = E_{17}$ and $E = E_{40}$ have $E(\bbQ)_{\tors} \cong \bbZ/4\bbZ$, it follows that if $E$ gains additional torsion points after base change to $K$, then $E$ necessarily gains a $K$-rational 2-torsion point, hence the full torsion subgroup $E(K)_{\tors}$ is precisely $\bbZ/2\bbZ \oplus \bbZ/4\bbZ$. It remains, then, to find the fields of definition for $E_{17}[2]$ and $E_{40}[2]$.

One can easily verify that $E_{17}[2]$ consists of $\infty$, $(1,-1)$, and the two points $(x,-1/2(x+1))$ with $4x^2 + x - 1 = 0$. Therefore $E_{17}$ attains full 2-torsion over $K = \bbQ(\sqrt{17})$. The 2-torsion on $E_{40}$ is perhaps more apparent: $E_{40}[2]$ consists of the points $\infty$, $(1,0)$, and $(x,0)$ with $x^2 + x - 1 = 0$. Hence $E_{40}$ attains full 2-torsion over $K = \bbQ(\sqrt{5})$.
\end{proof} 

To say that there exists a parameter $c \in K$ such that $G(f_c,K)$ contains a subgraph isomorphic to $G$ is equivalent to saying that $U_1(G)$ has a $K$-rational point. With this in mind, we now apply Theorem~\ref{thm:gen_one_torsion} to show the following:

\begin{prop}\label{prop:gen1}
Let $G$ be an admissible graph for which $X_1(G)$ has genus $1$, and let $K$ be a quadratic field.
	\begin{enumerate}
	\item If $\rk X_1(G)(K) = 0$, then $G$ does not occur as a subgraph of $G(f_c,K)$ for any $c \in K$, unless $G \cong {\rm10(2,1,1)b}$, $K = \bbQ(\sqrt{-15})$, and $c = 3/16$.
	\item If $\rk X_1(G)(K) \ge 1$, then $G$ occurs as a subgraph of $G(f_c,K)$ for infinitely many $c \in K$.
	\end{enumerate}
\end{prop}

\begin{proof}
We begin by proving (B), so assume that $X_1(G)(K)$ has positive Mordell-Weil rank. Then the curve $X_1(G)$ necessarily contains infinitely many $K$-rational points. Since $U_1(G)$ is open in $X_1(G)$, this implies that the set $U_1(G)(K)$ is infinite, hence the graph $G$ occurs as a subgraph of $G(f_c,K)$ for infinitely many parameters $c \in K$.

We now prove (A), so we suppose that $\rk X_1(G)(K) = 0$. If also $X_1(G)(K)_{\tors} = X_1(G)(\bbQ)_{\tors}$, then necessarily $X_1(G)(K) = X_1(G)(\bbQ)$. For each genus $1$ dynamical modular curve, the set $U_1(G)(\bbQ)$ is empty by \cite{poonen:1998}, so in this case $U_1(G)(K)$ is empty as well. Therefore the graph $G$ never occurs as a subgraph of $G(f_c,K)$ for any $c \in K$.

It remains to consider the case that $X_1(G)(K)_{\tors} \supsetneq X_1(G)(\bbQ)_{\tors}$.
We consider the graphs in the same order as in Theorem~\ref{thm:gen_one_torsion}, where all of the corresponding torsion subgroups are described.

If $G = {\rm8(1,1)b}$ or $G = {\rm8(2)b}$, then $X_1(G)(K)_{\tors} = X_1(G)(\bbQ)_{\tors}$ for all quadratic fields $K$, so the proposition holds for these graphs.

Now let $G = {\rm10(2,1,1)b}$. The only quadratic fields over which $X_1(G)$ gains torsion points are $K = \bbQ(\sqrt{d})$ with $d \in \{-15,-3,5\}$, and $X_1(G)$ has rank $0$ over all three of these fields. The four additional points on $X_1(G)(\bbQ(\sqrt{-15}))$ correspond to $c = 3/16$, in which case $G(f_c,\bbQ(\sqrt{-15})) \cong {\rm10(2,1,1)b}$, giving us the unique exception in the statement of the proposition. The four additional points on $X_1(G)(\bbQ(\omega))$ correspond to $c = 0$ and $c = -3/4$, for which $G(f_c,\bbQ(\omega))$ is isomorphic to 7(2,1,1)a and 6(1,1), respectively. Of the four additional points on $X_1(G)(\bbQ(\sqrt{5}))$, two are points at infinity, and the other two correspond to $c = -2$, in which case $G(f_c,\bbQ(\sqrt{5})) \cong {\rm9(2,1,1)}$.

In the case $G = {\rm10(2,1,1)a}$, the only quadratic field over which $X_1(G)$ gains torsion points is $\bbQ(\sqrt{17})$. However, $X_1(G)$ has rank $1$ over $\bbQ(\sqrt{17})$.

We now consider $G = {\rm8(1,1)a}$. In this case, $X_1(G)(K)_{\tors}$ is strictly larger than $X_1(G)(\bbQ)_{\tors}$ only for $K = \bbQ(\sqrt{d})$ with $d \in \{-3,-1,3\}$, and the rank of $X_1(G)(K)$ remains $0$ for each such $K$. For $d = -3$, the four additional points correspond to $c = 1/4$, where we have $G(f_c,\bbQ(\omega)) \cong {\rm4(1)}$. For $d = -1$, two of the additional points are points at infinity, while the other two correspond to $c = 0$, for which we have $G(f_c,\bbQ(i)) \cong {\rm5(1,1)b}$. When $d = 3$, the extra points on $X_1(G)$ correspond to $c = -2$, in which case $G(f_c,\bbQ(\sqrt{3})) \cong {\rm7(1,1)b}$.

Finally, let $G = {\rm8(2)a}$. The elliptic curve $X_1(G)$ only gains additional torsion over $K = \bbQ(\sqrt{5})$, and the rank is $0$ over $K$. The new points correspond to $c = -3/4$, where we actually have $G(f_c,\bbQ(\sqrt{5})) \cong {\rm6(1,1)}$.
\end{proof}

\subsection[Curves of genus $2$]{Curves of genus $2$}\label{sec:gen2}

In this section, we consider the Jacobians of each of the four genus $2$ dynamical modular curves listed in Proposition~\ref{prop:all_dyn_mod_curves}. As we did for the genus $1$ curves in the previous section, we explicitly determine the torsion subgroups of these four Jacobians over all quadratic extensions $K/\bbQ$.

Three of the four genus $2$ dynamical modular curves are also classical modular curves:
	\begin{equation}\label{eq:X1(1,3)}
		\begin{split}
			X_1(4) &\cong \Xell_1(16) : \hspace{5mm} y^2 = f_{16}(x) := -x(x^2 + 1)(x^2 - 2x - 1)\\
			X_1(1,3) &\cong \Xell_1(18) : \hspace{5mm} y^2 = f_{18}(x) := x^6 + 2x^5 + 5x^4 + 10x^3 + 10x^2 + 4x + 1\\
			X_1(2,3) &\cong \Xell_1(13) : \hspace{5mm} y^2 = f_{13}(x) := x^6 + 2x^5 + x^4 + 2x^3 + 6x^2 + 4x + 1.
		\end{split}
	\end{equation}
These curves correspond to the graphs 8(4), 10(3,1,1), and 10(3,2), respectively. For $G = {\rm8(3)}$, which is generated by a point of portrait $(2,3)$, it was shown in \cite{poonen:1998} that $X_1(G) = X_1((2,3))$ is given by the equation
	\[ y^2 = x^6 - 2x^4 + 2x^3 + 5x^2 + 2x + 1. \]
Each of these four curves $X_1(\cdot)$ has rational points, but none of them lie on $U_1(\cdot)$---with the exception of $(1, \pm 3) \in U_1((2,3))(\bbQ)$, corresponding to $c = -29/16$, for which $G(f_{-29/16}, \bbQ) \cong {\rm8(3)}$.
\begin{rem}
Since the notation is so similar, we pause to emphasize that $X_1(2,3)$ parametrizes maps $f_c$ together with marked points of period $2$ and $3$, respectively, while $X_1((2,3))$ parametrizes maps $f_c$ together with a single marked point of portrait $(2,3)$.
\end{rem}
The rational Mordell-Weil groups of the classical modular Jacobians $\Jell_1(N)$ with $N \in \{13,16,18\}$ are well known, and the same was computed for $J_1((2,3))$ in \cite{poonen:1998}:
	\begin{align*}
		J_1(4)(\bbQ) = \Jell_1(16)(\bbQ) &\cong \bbZ/2\bbZ \oplus \bbZ/10\bbZ\\
		J_1(1,3)(\bbQ) = \Jell_1(18)(\bbQ) &\cong \bbZ/21\bbZ\\
		J_1(2,3)(\bbQ) = \Jell_1(13)(\bbQ) &\cong \bbZ/19\bbZ\\
		J_1((2,3))(\bbQ) &\cong \bbZ,
	\end{align*}
We now determine over which quadratic fields $K$ these Jacobians gain new torsion points.

\begin{thm}\label{thm:gen_two_torsion}
Let $d$ be a squarefree integer, and let $K = \bbQ(\sqrt{d})$.
	\begin{enumerate}
	\item
		\[ J_1(4)(K)_{\tors} \cong
		\begin{cases}
			(\bbZ/2\bbZ)^2 \oplus \bbZ/10\bbZ, &\mbox{ if } d \in \{-1,2\};\\
			\bbZ/2\bbZ \oplus \bbZ/10\bbZ, &\mbox{ otherwise}.
		\end{cases}
		\]
	\item 
		\[ J_1(1,3)(K)_{\tors} \cong
		\begin{cases}
			\bbZ/3\bbZ \oplus \bbZ/21\bbZ, &\mbox{ if } d = -3;\\
			\bbZ/21\bbZ, &\mbox{ otherwise}.
		\end{cases}
		\]
	\item
		\[ J_1(2,3)(K)_{\tors} \cong \bbZ/19\bbZ. 
		\]
	\item
		\[
			J_1((2,3))(K)_\tors \cong 0.
		\]
	\end{enumerate}
\end{thm}

\begin{rem}
The torsion subgroups of the Jacobians of classical modular curves of genus 2 have been computed over certain quadratic fields---including $\bbQ(i)$ and $\bbQ(\omega)$---in \cite{kamienny/najman:2012,najman:2010,najman:2011}. Here, we compute the torsion subgroups over all quadratic fields simultaneously for dynamical modular curves of genus 2, just as we did in the genus 1 case in \textsection \ref{sec:gen1}.
\end{rem}

The most difficult case of Theorem~\ref{thm:gen_two_torsion} is (B). In order to prove (B), we will require two lemmas concerning the curve $X_1(1,3)$. Recall that for a divisor $D$ on a curve $X$, the {\it Riemann-Roch space of $D$} is the space
	\[
		\calL(D) := \{f \in \bbC(X) : (f) + D \text{ is effective}\},
	\]
and the {\it complete linear system} $|D|$ is the set of all effective divisors linearly equivalent to $D$; that is, the set of effective divisors $E$ for which $[E - D] = [0] = \calO$.
\begin{lem}\label{lem:3infty}
Let $X = X_1(1,3)$, given by the model $y^2 = f_{18}(x)$ from \eqref{eq:X1(1,3)}. Then $[3P - 3\infty^+] = \calO$ if and only if $P = \infty^+$.
\end{lem}

\begin{proof}
One direction is immediate, so we suppose $P \ne \infty^+$ and show that $[3P - 3\infty^+] \ne \calO$.

First, take $P = \infty^-$. We verify in \cite[{\tt main.txt}]{doyle:cyclotomic} that the point $[\infty^- - \infty^+] \in J_1(2,3)(\bbQ)$ is a point of order 21, which means that $[3\infty^- - 3\infty^+] = 3[\infty^- - \infty^+] \ne \calO$.

Next, we observe that if $P$ is a Weierstrass point, then
	\[ [3P - 3\infty^+] = [P + \infty^- - 2\infty^+] + [2P - \infty^+ - \infty^-] = [P + \infty^- - 2\infty^+]. \]
Since $\infty^+$ is not a Weierstrass point, there is no rational function of degree $2$ with a pole only at $\infty^+$, so $[P + \infty^- - 2\infty^+]$---and therefore $[3P - 3\infty^+]$---is nonzero.

Finally, suppose $P$ is a finite, non-Weierstrass point on $X$; write $P = (x_0,y_0)$ with $y_0 \ne 0$. Suppose for contradiction that $[3P - 3\infty^+] = \calO$, and let $g$ be a rational function on $X$ with zero divisor $3P$ and pole divisor $3\infty^+$. Then $g$ lies in $\calL(3\infty^+)$, which has dimension $2$ by the Riemann-Roch theorem. The constant functions certainly lie in $\calL(3\infty^+)$, and we claim that the function $h := y + x^3 + x^2 + 2x$ also lies in $\calL(3\infty^+)$, so that $\calL(3\infty^+) = \langle 1, h \rangle$. In other words, we claim that the only pole of $h$ is a triple pole at $\infty^+$. Certainly $h$ has no finite poles. To better understand the behavior of $h$ at infinity, we cover $X$ by the affine patches $y^2 = f_{18}(x)$ and $v^2 = u^6 f_{18}(1/u)$, with the identifications $x = 1/u$ and $y = v/u^3$ (as described in \textsection \ref{sec:hyp}). The two points $\infty^{\pm}$ on $X$ are given by $(u,v) = (0,\pm 1)$. We rewrite $h$ in terms of $u$ and $v$ to get
	\begin{equation}\label{eq:h_uv}
	h = \frac{v + 2u^2 + u + 1}{u^3}.
	\end{equation}
Certainly $h$ has a triple pole at $(u,v) = (0,1)$, since $u = 1/x$ is a uniformizer at $\infty^+$. On the other hand, multiplying each of the numerator and denominator of \eqref{eq:h_uv} by $v - (2u^2 + u + 1)$ yields
	\[ h = \frac{u^3 + 4u^2 + 6u + 6}{v - (2u^2 + u + 1)}, \]
which now visibly does not have a pole at $(u,v) = (0,-1)$. Therefore $h \in \calL(3\infty^+)$, as claimed.

It follows that the function $g$ may be written as $a + b(y + x^3 + x^2 + 2x)$ for some scalars $a$ and $b$. Since $g$ must be nonconstant, we must have $b \ne 0$; scaling by $1/b$, we may assume $g$ is of the form
	\[ g = y + x^3 + x^2 + 2x + A, \]
which we rewrite as
	\[
		g = \frac{y^2 - (x^3 + x^2 + 2x + A)^2}{y - (x^3 + x^2 + 2x + A)} = -\frac{p(x)}{y - (x^3 + x^2 + 2x + A)},
	\]
where
	\[ p(x) := 2(A-3)x^3 + 2(A-3)x^2 + 4(A-1)x + (A+1)(A-1). \]
Since $P$ is not a Weierstrass point, $x - x_0$ is a uniformizer at $P$; since $g$ vanishes to order $3$ at $P$, this means that $(x-x_0)^3$ must divide $p(x)$. Thus each of $p(x)$ and $p'(x)$ has a multiple root, so
\[
	\disc(p) = -4 (A-1) (A-3) \left(27 A^4-118 A^3+180 A^2-42 A+17\right) = 0
\]
and
\[
	\disc(p') = -16(A - 3)(5A - 3) = 0.
\]
This forces $A = 3$, which contradicts the fact that $p(x)$ must have degree $3$. Having exhausted all possibilities for $P \ne \infty^+$, we have completed the proof.
\end{proof}

\begin{lem}\label{lem:3torsion}
Let $X = X_1(1,3)$ and $J = J_1(1,3)$. The $3$-torsion subgroup $J[3]$ contains only nine points of degree at most $2$ over $\bbQ$, all of which are defined over $\bbQ(\omega)$.
\end{lem}

\begin{proof}
Suppose $\{P,Q\}$ is a point of order 3 on $J$. This means that
	\begin{equation}\label{eq:3-torsion}
	[3P + 3Q - (3 \infty^+ + 3\infty^-)] = 3[P + Q - \infty^+ - \infty^-] = \calO,
	\end{equation}
so there is a function $g$ on $X$ whose divisor is $(g) = 3P + 3Q - (3 \infty^+ + 3\infty^-)$.

We first show that neither $P$ nor $Q$ may be a point at infinity. Suppose to the contrary that $Q = \infty^-$. (There is no loss of generality here: The pair $\{P,Q\}$ is unordered, so we are free to switch $P$ and $Q$, and if $\{P,\infty^+\}$ is a 3-torsion point, then so is $-\{P,\infty^+\} = \{\iota P, \infty^-\}$.) Then $\calO = [3P + 3Q - (3 \infty^+ + 3\infty^-)] = [3P - 3\infty^+]$. However, by Lemma~\ref{lem:3infty} this implies $P = \infty^+$, which means that $\{P,Q\} = \{\infty^+,\infty^-\} = \calO$, hence $\{P,Q\}$ does not have order $3$.

We now show that neither $P$ nor $Q$ may be a Weierstrass point. Suppose for contradiction that $P$ is a Weierstrass point. (Again, we lose no generality in doing so since $\{P,Q\}$ is unordered.) Then
	\begin{align*}
	\calO = [3P + 3Q - 3\infty^+ - 3\infty^-]
		= \{P,P\} + \{Q,Q\} + \{P,Q\}
		= \{Q,Q\} + \{P,Q\},
	\end{align*}
since $P$ is assumed to be a Weierstrass point. This implies that
	\[
		\{P,Q\} = -\{Q,Q\} = \{\iota Q, \iota Q\},
	\]
hence $P = \iota Q = Q$. It follows that $\{P,Q\} = \{P,P\} = \calO$, so $\{P,Q\}$ does not have order $3$.

Since neither $P$ nor $Q$ is a point at infinity, there is no cancellation in the difference $3P + 3Q - 3\infty^+ - 3\infty^-$, so the function $g$ must have zero divisor equal to $3P + 3Q$ and pole divisor equal to $D := 3\infty^+ + 3\infty^-$. By Riemann-Roch, $\dim \calL(D) = 5$; since the set
	\[ \{1, x, x^2, x^3, y\} \]
is a linearly independent set of elements of $\calL(D)$, it must be a basis. Therefore there exist scalars $a,b,c,d,e \in \bbC$ for which
	\[ g = ay + bx^3 + cx^2 + dx + e. \]
We claim that $a \ne 0$. Indeed, if $a = 0$, then the set of points on $X$ for which $g = 0$ is 
	\[ \calS := \{(x,\pm \sqrt{f_{18}(x)}) : bx^3 + cx^2 + dx + e = 0\}. \]
Since $(g) = 3(P + Q - \infty^+ - \infty^-)$, $\calS$ contains only two points. Thus either $g = 0$ has a single solution $x_0$, or $g = 0$ has two distinct solutions $x_1$ and $x_2$ with $f_{18}(x_1) = f_{18}(x_2) = 0$. In the former case, $P$ and $Q$ are hyperelliptic conjugates, so $\{P,Q\} = \calO$; in the latter, $P$ and $Q$ are distinct Weierstrass points, which we have already ruled out. In either case, $\{P,Q\}$ is not a point of order 3, so we must have $a \ne 0$; dividing by $a$ if necessary, we take $g$ to be of the form
	\[ g = y - (Ax^3 + Bx^2 + Cx + D), \]
which we rewrite as
	\begin{align*}
	g &= \frac{y^2 - (Ax^3 + Bx^2 + Cx + D)^2}{y + (Ax^3 + Bx^2 + Cx + D)}\\
		&= -\frac{q(x)}{y + (Ax^3 + Bx^2 + Cx + D)},
	\end{align*}
where
	\begin{align}\label{eq:3tors_poly1}
		\begin{split}
		q(x) = (A+1)(A-1)x^6 &+ 2(AB - 1)x^5 + (2AC + B^2 - 5)x^4 \\
			& + 2(AD + BC - 5)x^3 + (2BD + C^2 - 10)x^2 \\
			& + 2(CD - 2)x + (D+1)(D-1).
		\end{split}
	\end{align}
The function $q(x)$ must vanish to order 3 at each of $P = (x_1,y_1)$ and $Q = (x_2,y_2)$. Since $P$ and $Q$ are not Weierstrass points, $(x-x_1)$ and $(x-x_2)$ are uniformizers at $P$ and $Q$, respectively. Thus $(x-x_1)^3(x-x_2)^3$ must divide $q(x)$, hence $(A + 1)(A - 1) \ne 0$ and
	\begin{equation}\label{eq:3tors_poly2}
	q(x) = (A+1)(A-1)(x^2 - tx + n)^3,
	\end{equation}
where $t = x_1 + x_2$ and $n = x_1x_2$. Equating the coefficients of the expressions for $q(x)$ given in \eqref{eq:3tors_poly1} and \eqref{eq:3tors_poly2} yields the following system of equations:

	\begin{equation}\label{eq:3tors_system}
		\left\{
		\begin{split}
			\hfill 2(AB - 1) &= -3(A+1)(A-1)t\\
			\hfill 2AC + B^2 - 5 &= 3 (A+1)(A-1)(t^2 + n)\\
			\hfill 2(AD + BC - 5) &= -(A+1)(A-1)t(t^2 + 6n)\\
			\hfill 2BD + C^2 - 10 &= 3(A+1)(A-1)n(t^2 + n)\\
			\hfill 2(CD - 2) &= -3(A+1)(A-1)tn^2\\
			\hfill (D+1)(D-1) &= (A+1)(A-1)n^3
		\end{split}
		\right.
	\end{equation}
The system \eqref{eq:3tors_system} defines a $0$-dimensional scheme $S \subseteq \bbA^6$. In \cite[{\tt main.txt}]{doyle:cyclotomic}, we find all 80 points of $S(\QQbar)$, each of which corresponds to a point of order 3 in $J(\QQbar)$.

Now, in order for $\{P,Q\}$ to be a \emph{quadratic} point on $J$, say defined over the quadratic field $K$, either $P$ and $Q$ must both be defined over $K$, or $P$ and $Q$ must be Galois conjugates defined over some quadratic extension $L/K$. In either case, the parameters $t$ and $n$ must both lie in $K$. The only points in $S(\QQbar)$ with $[\bbQ(t,n):\bbQ] \le 2$ are the eight points satisfying
	\[
		(t,n) \in \{(-1,1),(-2(\omega + 1),\omega),(-2(\omega^2 + 1),\omega^2)\},
	\]
all of which are defined over $\bbQ(\omega)$. Therefore the only quadratic field over which $J$ gains additional 3-torsion is $\bbQ(\omega)$, and over this field there are a total of eight points (two of which are $\bbQ$-rational) of order $3$. The lemma now follows.
\end{proof}

\begin{proof}[Proof of Theorem~\ref{thm:gen_two_torsion}]
For (A), we note that 3 is a prime of good reduction for $J_1(4)$, and that
	\[ J_1(4)(\bbF_{3^2}) \cong (\bbZ/2\bbZ)^3 \oplus \bbZ/10\bbZ. \]
Moreover, since 5 is a prime of good reduction and $\# J_1(4)(\bbF_{5^2}) = 2^7 \cdot 5$, $J_1(4)(K)$ cannot have 3-torsion. Hence
	\[ J_1(4)(K)_{\tors} \hookrightarrow (\bbZ/2\bbZ)^3 \oplus \bbZ/10\bbZ. \]
Since $J_1(4)(\bbQ)_{\tors} \cong \bbZ/2\bbZ \oplus \bbZ/10\bbZ$, the only way for $J_1(4)(K)_{\tors}$ to be strictly larger than $J_1(4)(\bbQ)_{\tors}$ is for $J_1(4)$ to gain a 2-torsion point upon base change from $\bbQ$ to $K$. By Lemma~\ref{lem:2torsion}, the 2-torsion points are the points supported on the Weierstrass locus of $X_1(4)$. The Weierstrass points are $\infty$ and the points
	\[ P := (0,0), \ Q^{\pm} := (\pm i,0), \text{ and } R^{\pm} := (1 \pm \sqrt{2},0). \]
The sixteen points in $J_1(4)[2]$ are therefore those appearing in Table~\ref{tab:2torsion}. Hence the only quadratic fields over which $J_1(4)$ gains additional torsion are $\bbQ(i)$ and $\bbQ(\sqrt{2})$, and over each of these fields the torsion subgroup is isomorphic to $(\bbZ/2\bbZ)^2 \oplus \bbZ/10\bbZ$.

\begin{table}
	\centering
	\renewcommand{\arraystretch}{1.5}
	\caption{The 2-torsion points on $J_1(4)$}
	\label{tab:2torsion}
	\begin{tabular}{|c|c c c c|}
	\hline
	Field of definition & \multicolumn{4}{|c|}{Points} \\ \hline
	$\bbQ$ & $\calO$ & $\{\infty, P\}$ & $\{Q^+,Q^-\}$ & $\{R^+,R^-\}$ \\ \hline
	$\bbQ(i)$ & $\{\infty,Q^+\}$ & $\{\infty,Q^-\}$ & $\{P,Q^+\}$ & $\{P,Q^-\}$\\ \hline
	$\bbQ(\sqrt{2})$ & $\{\infty,R^+\}$ & $\{\infty,R^-\}$ & $\{P,R^+\}$ & $\{P,R^-\}$\\ \hline
	$\bbQ(i,\sqrt{2})$ & $\{Q^+,R^+\}$ & $\{Q^+,R^-\}$ & $\{Q^-,R^+\}$ & $\{Q^-,R^-\}$\\ \hline
	\end{tabular}
\end{table}

Next, we consider part (B). Since $J_1(1,3)$ has good reduction at the primes 5 and 11, we compute
	\begin{align*}
	J_1(1,3)(\bbF_{5^2}) &\cong (\bbZ/3\bbZ)^2 \oplus (\bbZ/7\bbZ)^2 , \  \\
	J_1(1,3)(\bbF_{11^2}) &\cong (\bbZ/4\bbZ)^2 \oplus (\bbZ/3\bbZ)^2 \oplus \bbZ/7\bbZ \oplus \bbZ/13\bbZ.
	\end{align*}
Therefore
	\[ J_1(1,3)(K)_{\tors} \hookrightarrow (\bbZ/3\bbZ)^2 \oplus \bbZ/7\bbZ = \bbZ/3\bbZ \oplus \bbZ/21\bbZ. \]
Since $J_1(1,3)(\bbQ) \cong \bbZ/21\bbZ$, the only way for $J_1(1,3)$ to gain torsion points over a quadratic field $K$ is to gain a point of order 3, in which case the full torsion subgroup is $\bbZ/3\bbZ \oplus \bbZ/21\bbZ$. We know from Lemma~\ref{lem:3torsion} that the only quadratic field over which $J_1(1,3)$ admits additional $K$-rational points of order 3 is $K = \bbQ(\omega)$, and (B) now follows.

For parts (C) and (D), we observe that 3 and 5 are both primes of good reduction for $J_1(2,3)$ and $J_1((2,3))$, and that
	\begin{alignat*}{4}
		\#J_1(2,3)(\bbF_{3^2}) &= 3 \cdot 19 & \hspace{10mm} \#J_1((2,3))(\bbF_{3^2}) &= 3^4;\\
		\#J_1(2,3)(\bbF_{5^2}) &= 19^2 & \#J_1((2,3))(\bbF_{5^2}) &= 19 \cdot 43.
	\end{alignat*}
Therefore $J_1(2,3)(K)_{\tors} \hookrightarrow \bbZ/19\bbZ$ and $J_1((2,3))(K)_{\tors} = 0$. Since $J_1(2,3)(\bbQ) \cong \bbZ/19\bbZ$, this proves (C) and (D).
\end{proof}

\begin{prop}\label{prop:gen2}
Let $G$ be an admissible graph for which $J_1(G)$ has genus $2$, and let $K$ be a quadratic field. Suppose $\rk J_1(G)(K) = \rk J_1(G)(\bbQ)$.
	\begin{enumerate}
		\item If $G$ is isomorphic to 8(4), 10(3,1,1), or 10(3,2), then $G$ does not occur as a subgraph of $G(f_c,K)$ for any $c \in K$.
		\item If $G = \rm 8(3)$, then the only $c \in K$ for which $G(f_c,K)$ contains a subgraph isomorphic to $G$ is $c = -29/16$, in which case $G(f_c,\bbQ) \cong {\rm8(3)}$.
	\end{enumerate}
\end{prop}

\begin{proof}
We begin with statement (A). If $G$ is one of the graphs 8(4), 10(3,1,1), or 10(3,2), then $\rk J_1(G)(\bbQ) = 0$, in which case the conditions
	\[\rk J_1(G)(K) = \rk J_1(G)(\bbQ) \ \mbox{ and } \ J_1(G)(K)_{\tors} = J_1(G)(\bbQ)_{\tors} \]
automatically imply that $J_1(G)(K) = J_1(G)(\bbQ)$. Since $X_1(G)(\bbQ) \ne \emptyset$ for each of these three graphs $G$, Proposition~\ref{prop:same_jac} immediately gives us $X_1(G)(K) = X_1(G)(\bbQ)$. Since $U_1(G)(\bbQ) = \emptyset$ for each $G$ (see \cite{poonen:1998,morton:1998}), we conclude that (A) holds if $J_1(G)(K)_\tors = J_1(G)(\bbQ)_\tors$. It remains to consider those quadratic fields $K$ for which $J_1(G)(K)_\tors \supsetneq J_1(G)(\bbQ)_\tors$.

We begin by considering $G = {\rm8(4)}$. By Theorem~\ref{thm:gen_two_torsion}, $J_1(G) = J_1(4)$ only gains torsion points over $K = \bbQ(i)$ and $K = \bbQ(\sqrt{2})$, and $J_1(G)$ still has rank $0$ over those two fields. Over each of these fields $K$, we can explicitly determine all forty elements of $J_1(G)(K)$, and we find no {\it non-trivial} points of the form $\{P,P\}$ with $P \in X_1(G)(K) \setminus X_1(G)(\bbQ)$. This means that the only additional $K$-rational points on $X_1(G)$ are the Weierstrass points: $(\pm i, 0)$ over $\bbQ(i)$, and $(1 \pm \sqrt{2},0)$ over $\bbQ(\sqrt{2})$. However, the points $(\pm i,0)$ correspond to $c = (\mp2i + 1)/4$, for which we have $G(f_c,\bbQ(i)) \cong {\rm4(1,1)}$, and the points $(1 \pm \sqrt{2},0)$ correspond to $c = -5/4$, for which we have $G(f_c,\bbQ(\sqrt{2})) \cong {\rm4(2)}$.

Now let $G = {\rm10(3,1,1)}$. The Jacobian $J_1(G) = J_1(1,3)$ only gains additional torsion over the quadratic field $K = \bbQ(\omega)$. As in the previous case, we can explicitly find all 63 points on $J_1(G)(K)$, which allows us to completely determine $X_1(G)(K)$. The only new points on $X_1(G)(K)$ are $(\omega, \pm(\omega - 1))$ and their Galois conjugates. These correspond to $c =
(1 + 3\omega)/4$ and its conjugate, for which we have $G(f_c,K) \cong {\rm4(1,1)}$.

For $G = {\rm10(3,2)}$, the torsion subgroup of $J_1(G) = J_1(2,3)$ is unchanged upon base change to any quadratic field $K$, so we are already done in this case.

We now prove (B), so let $G = {\rm8(3)}$. In this case, we have $\rk J_1(G)(\bbQ) = 1$, so assume $K$ is a quadratic field with $\rk J_1(G)(K) = 1$. Since $J_1(G)(K)_\tors$ is trivial for all quadratic fields $K$, Corollary~\ref{cor:saturation_samerank} tells us that $J_1(G)(K) = J_1(G)(\bbQ)$; since $X_1(G)(\bbQ)$ is nonempty, it follows that Proposition~\ref{prop:same_jac} that $X_1(G)(K) = X_1(G)(\bbQ)$. As shown in \cite[\textsection 4]{poonen:1998}, the only points on $U_1(G)(\bbQ)$---and, therefore, the only points on $U_1(G)(K)$---correspond to $c = -29/16$, in which case we have $G(f_c,\bbQ) \cong 8(3)$.
\end{proof}

\subsection[The curve $X_0(5)$]{The curve $X_0(5)$}\label{sec:X0(5)}

It is shown in \cite{flynn/poonen/schaefer:1997} that if $c \in \bbQ$, then $f_c$ cannot admit rational points of period 5. Rather than attempting to directly find all rational points on the genus $14$ curve $X_1(5)$, the authors of \cite{flynn/poonen/schaefer:1997} work with the quotient curve $X_0(5)$, which parametrizes maps $f_c$ together with a marked {\it cycle} of length $5$. The model given in \cite{flynn/poonen/schaefer:1997} for $X_0(5)$ is
	\begin{equation}\label{eq:X0(5)}
	y^2 = x^6  + 8x^5 + 22x^4 + 22x^3 + 5x^2 + 6x + 1.
	\end{equation}
They show that $\rk J_0(5)(\bbQ) = 1$, and then they determine that
	\[ X_0(5)(\bbQ) = \{(0,\pm 1), (-3,\pm 1), \infty^{\pm} \} \]
using a version of the Chabauty-Coleman method for genus $2$ curves developed by Flynn \cite{flynn:1997}. They conclude that the only values of $c \in \bbQ$ for which $f_c$ has a rational 5-cycle (i.e., the cycle is Galois invariant as a set, but not necessarily pointwise) are $-2$, $-16/9$, and $-64/9$. However, for each such $c$ the corresponding points of period 5 generate a degree 5 extension of $\bbQ$.

Since $X_0(5)$ has genus $2$, we may apply the methods used in the previous section to compute the torsion subgroup of $J_0(5)(K)$ for quadratic fields $K$. From this information, we will deduce a sufficient condition for a quadratic field $K$ to contain no elements $c$ for which $f_c$ admits $K$-rational points of period 5.

\begin{prop}
Let $K$ be a quadratic field. Then
	\[ J_0(5)(K)_{\tors} = 0. \]
\end{prop}

\begin{proof}
The primes $p = 3$ and $p = 5$ are primes of good reduction for the curve $X$ given in \eqref{eq:X0(5)}, which is birational to $X_0(5)$. Letting $J := \Jac(X)$, a computation in \cite[{\tt main.txt}]{doyle:cyclotomic} shows that
	\[ \#J(\bbF_{3^2}) = 3^4 \text{\ \ and\ \ } \#J(\bbF_{5^2}) = 29 \cdot 41. \]
As before, if $\frakp$ is any prime in $\OK$ lying above the rational prime $p$, then $\bbF_{\frakp} \hookrightarrow \bbF_{p^2}$ and, therefore, $J(\bbF_{\frakp}) \hookrightarrow J(\bbF_{p^2})$. Since $\#J(\bbF_{3^2})$ and $\#J(\bbF_{5^2})$ are coprime, we conclude that $J(K)_{\tors} = 0$.
\end{proof}

\begin{cor}\label{cor:5cycle}
Let $K$ be a quadratic field. If $\rk J_0(5)(K) = 1$, then there is no element $c \in K$ for which $f_c$ admits a $K$-rational point of period 5.
\end{cor}

\begin{proof}
Let $X := X_0(5)$ and $J := J_0(5)$. If $\rk J(K) = 1$, then we have $\rk J(K) = \rk J(\bbQ)$ and $J(K)_{\tors} = 0$, hence $J(K) = J(\bbQ)$ by Corollary~\ref{cor:saturation_samerank}. Since $X$ has rational points, we conclude from Proposition~\ref{prop:same_jac} that $X(K) = X(\bbQ)$, so the only $c \in K$ such that $f_c$ has a $K$-rational 5-cycle are $c \in \{-2,-16/9,-64/9\}$. However, as mentioned above, the points of period 5 must actually lie in a degree 5 extension of $K$, so $f_c$ has no $K$-rational points of period 5.
\end{proof}

\begin{thm}\label{thm:5cycle_cyclotomic}
Let $K$ be the field $\bbQ(i)$ or $\bbQ(\omega)$. There is no element $c \in K$ for which $f_c$ admits a $K$-rational point of period 5.
\end{thm}

\begin{proof}
In both cases, we have $X_0(5)(K) = X_0(5)(\bbQ)$: For $K = \bbQ(\omega)$, this follows from the fact that the twist of $X_0(5)$ by $-3$ has rank $0$ over $\bbQ$, so the rank of $J_0(5)(K)$ is equal to $1$, and therefore Corollary~\ref{cor:5cycle} applies. For $K = \bbQ(i)$, the rank of $J_0(5)$ is equal to $2$, so the proof of this statement requires a more involved argument. We therefore defer the proof, which uses a Chabauty-Coleman-style calculation, to Appendix~\ref{app:X0(5)}.
\end{proof}

\section[Preperiodic points over cyclotomic quadratic fields]{Preperiodic points over cyclotomic quadratic fields}\label{sec:specific}

We now move from making general statements that hold over arbitrary quadratic fields to giving results over two particular quadratic fields---namely, the cyclotomic quadratic fields. Our main result is a conditional classification result like that of Poonen \cite{poonen:1998}, but over these two quadratic extensions of $\bbQ$ rather than over $\bbQ$ itself.

We begin by restricting the cycle structures that can appear for a graph $G(f_c,K)$ with $K$ a quadratic cyclotomic field and $c \in K$.

\begin{lem}\label{lem:structures}
Let $K$ be the field $\bbQ(i)$ or $K = \bbQ(\omega)$, let $c \in K$, and assume $f_c$ does not admit points of period greater than $5$. If $G(f_c,K)$ is strongly admissible, then the cycle structure of $G(f_c,K)$ is $(1,1)$, $(2)$, $(3)$, or $(1,1,2)$
\end{lem}

\begin{rem}
It was shown by Erkama in \cite{erkama:2006} that if $c \in \bbQ(i)$, then $f_c$ cannot have $\bbQ(i)$-rational points of period 4. His proof uses different techniques from ours, including an interesting {\it $2$}-dimensional dynamical system that models iteration of the family $f_c$.
\end{rem}

\begin{proof}[Proof of Lemma~\ref{lem:structures}]
We find in \cite[{\tt main.txt}]{doyle:cyclotomic} that $\rk J_1(4)(K) = 0$ for both fields $K$, thus by Proposition~\ref{prop:gen2} there is no $c \in K$ for which $f_c$ has a $K$-rational point of period $4$. Further, Theorem~\ref{thm:5cycle_cyclotomic} says that there is no $c \in K$ for which $f_c$ has a $K$-rational point of period $5$. Thus, we now assume that $f_c$ has no $K$-rational points of period greater than $3$.

It follows from the results in \cite[\textsection 4]{doyle:2018quad} that if $K$ is {\it any} quadratic field, $c \in K$, and $G(f_c,K)$ is strongly admissible with no cycles of length greater than $3$, then the cycle structure of $G(f_c,K)$ must be one of the following:
\[
	(1,1),\ (2),\ (3),\ (1,1,2),\ (1,1,3),\ (2,3).
\]
It therefore remains to show that for both fields $K$ under consideration, if $c \in K$ admits a $K$-rational point of period 3, then it has no $K$-rational points of period 1 or 2. Indeed, \cite[{\tt main.txt}]{doyle:cyclotomic} shows that for both fields $K$, $\rk J_1(1,3)(K) = \rk J_1(2,3)(K) = 0$, so the result follows from Proposition~\ref{prop:gen2}.
\end{proof}

\begin{figure}
\[ \footnotesize
	\xymatrix{
		G_1 \ar[d]\ar[dr] & \text{10(1,1)a/b} \ar[d] & \text{10(2,1,1)b} \ar[ddl]\ar[dr] & \fbox{10(2,1,1)a} \ar[d]\ar[ddrr] & G_2 \ar[dr]\ar[drr] & G_3 \ar[dr] & \text{10(2)} \ar[d] & \text{10(3)a/b} \ar[d]\\
		\text{8(1,1)a} \ar[dr] & \text{8(1,1)b} \ar[d] & & \dbox{\fbox{8(2,1,1)}} \ar[ddll]\ar[ddrr] & & \dbox{8(2)a} \ar[d] & \text{8(2)b} \ar[dl] & \dbox{\fbox{8(3)}} \ar[d]\\
		& \dbox{\fbox{6(1,1)}} \ar[d] & & & & \dbox{\fbox{6(2)}} \ar[d] & & \dbox{\fbox{6(3)}} \\
		& \dbox{\fbox{4(1,1)}} & & & & \dbox{\fbox{4(2)}} &
	}
\]
\caption{Strongly admissible graphs with at most ten vertices and cycle structure (1,1), (2), (3), or (1,1,2). There is a directed path from $G$ to $H$ if and only if $H \subset G$. A graph has a \fbox{solid} (resp., \protect\dbox{dashed}) box around it if it is realized as $G(f_c,K)$ over $K = \bbQ(i)$ (resp., $K = \bbQ(\omega)$).} 
\label{fig:directed_system}
\end{figure}
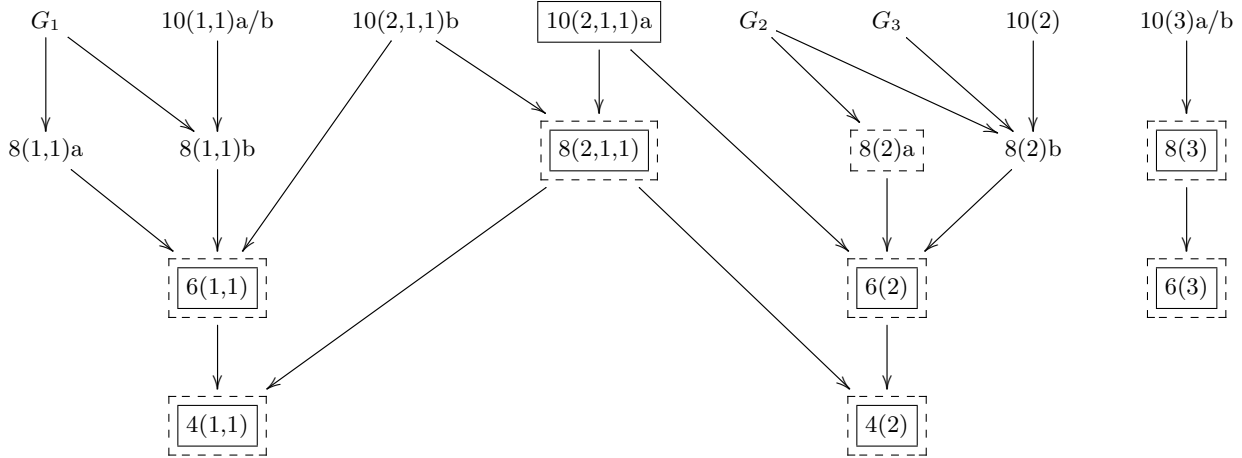

We now briefly sketch an outline of the proof of Theorem~\ref{thm:main}, which we have separated into Propositions~\ref{prop:main_i} and \ref{prop:main_omega} for $\bbQ(i)$ and $\bbQ(\omega)$, respectively. For each cyclotomic quadratic field $K$, certain small strongly admissible graphs $G$ do occur as $G(f_c,K)$ for some $c \in K$; the known such graphs appear in Appendix~\ref{app:all_data} and are indicated for convenience in Figure~\ref{fig:directed_system}. For reference, Figure~\ref{fig:directed_system} actually includes \textit{all} strongly admissible graphs with at most ten vertices and having cycle structures allowed by Lemma~\ref{lem:structures}.

For each cycle structure allowed by Lemma~\ref{lem:structures}, we then consider those strongly admissible graphs (with the given cycle structure) that are {\it minimal} among those not known to occur over $K$. For each such graph $G$, we show that any $K$-rational points on $X_1(G)(K)$ actually correspond to parameters $c \in K$ for which $G(f_c,K)$ is not isomorphic to $G$; i.e., we show that $U_1(G)(K)$ is empty.

This final step relies on the results of \textsection \ref{sec:quad}. For many of the arguments, we use the fact that the Jacobians of certain dynamical modular curves have rank $0$ over various quadratic fields. In every such case, the relevant rank computations were performed using the method of $2$-descent implemented in Magma's \texttt{RankBound} function.

We now prove Theorem~\ref{thm:main}, which we state as two separate propositions for convenience. As before, all graphs appear in Appendix~\ref{app:all_data}.

\begin{prop}\label{prop:main_i}
Let $K = \bbQ(i)$, and let $c \in K$. Suppose $f_c$ does not admit $K$-rational points of period greater than 5. Then $G(f_c,K)$ is isomorphic to one of the following fourteen graphs:
	\begin{center}
		\rm 0, 3(2), 4(1,1), 4(2), 5(1,1)a, 5(1,1)b, 5(2)a, 6(1,1), 6(2), 6(2,1), 6(3), 8(2,1,1), 8(3), 10(2,1,1)a.
	\end{center}
\end{prop}

\begin{proof}
Each of the graphs listed does occur over $\bbQ(i)$, as seen in Appendix~\ref{app:all_data}. Also, it follows from \cite[\textsection 5]{doyle/faber/krumm:2014} that 3(2), 5(1,1)a/b, 5(2)a, and 6(2,1) are the only graphs that are not strongly admissible but that may be realized as $G(f_c,K)$ for some $c \in K$. We henceforth assume that $G(f_c,K)$ is strongly admissible, and by Lemma~\ref{lem:structures} we may further assume that the cycle structure of $G(f_c,K)$ is (1,1), (2), (3), or (1,1,2). For such $c \in K$, it suffices to show that $G(f_c,K)$ is isomorphic to one of the following:
	\begin{center}
	0, 4(1,1), 4(2), 6(1,1), 6(2), 6(3), 8(2,1,1), 8(3), 10(2,1,1)a.
	\end{center}
By considering the system of graphs in Figure~\ref{fig:directed_system}, we must show the following:
	\begin{enumerate}
	\item The graph $G(f_c,K)$ does not contain a subgraph isomorphic to 8(1,1)a, 8(1,1)b, 8(2)a, 8(2)b, or 10(2,1,1)b.
	\item The graph $G(f_c,K)$ does not {\it properly} contain a subgraph isomorphic to 10(2,1,1)a.
	\item If $G(f_c,K)$ contains a subgraph isomorphic to 8(3), then $c = -29/16$ and $G(f_c,K) \cong$ 8(3).
	\end{enumerate}
	
Statement (A) holds by applying Proposition~\ref{prop:gen1} to each graph $G$ listed in part (A), since $X_1(G)$ has genus $1$ and $\rk X_1(G)(K) = 0$ for each such $G$.

For (B), we first note that the only 12-vertex strongly admissible graphs containing 10(2,1,1)a are 12(2,1,1)a, $G_4$, and $G_6$. The graph $G_6$ contains 10(2,1,1)b, so by (A) it cannot be a subgraph of $G(f_c,K)$. It remains to show, then, that $G(f_c,K)$ cannot contain 12(2,1,1)a or $G_4$. For 12(2,1,1)a, this follows from \cite[Cor. 3.36]{doyle/faber/krumm:2014}. It was shown in \cite[Prop. 5.10]{doyle:2018quad} that $X_1(G_4)$ has a model of the form
\[
	\left\{
	\begin{split}
		y^2 &= 2(x^3 + x^2 - x + 1)\\
		z^2 &= 5x^4 + 8x^3 + 6x^2 - 8x + 5,
	\end{split}
	\right.
\]
and that any finite quadratic point $(x,y,z)$ on $X_1(G_4)$ satisfies $x \in \bbQ$ and $y,z \notin \bbQ$. Therefore, a finite $K$-rational (but not $\bbQ$-rational) point on $X_1(G_4)$ yields a {\it rational} point on the twist
\begin{equation}\label{eq:twist-1}
	\left\{
	\begin{split}
		-y^2 &= 2(x^3 + x^2 - x + 1)\\
		-z^2 &= 5x^4 + 8x^3 + 6x^2 - 8x + 5.
	\end{split}
	\right.
\end{equation}
The curve $C$ defined by $-y^2 = 2(x^3 + x^2 - x + 1)$ is birational to the elliptic curve labeled 176B1 in \cite{cremona:1997}, which has a single rational point. Since $C$ has a rational point at infinity, $C$ has no finite rational points, hence there are no rational solutions to \eqref{eq:twist-1}. Therefore $G(f_c,K)$ cannot contain a graph isomorphic to $G_4$.

Finally, for (C), we note that for $G = {\rm8(3)}$ we have $\rk J_1(G)(K) = 1$, which means (by Proposition~\ref{prop:gen2}) that the only $c \in K$ with $G(f_c,K)$ containing 8(3) is $c = -29/16$, and a simple calculation verifies that in this case $G(f_c,K) \cong {\rm8(3)}$.
\end{proof}

\begin{prop}\label{prop:main_omega}
Let $K = \bbQ(\omega)$, and let $c \in K$. Suppose $f_c$ does not admit $K$-rational points of period greater than 5. Then $G(f_c,K)$ is isomorphic to one of the following thirteen graphs:
	\begin{center}
		\rm 0, 3(2), 4(1), 4(1,1), 4(2), 5(1,1)a, 6(1,1), 6(2), 6(3), 7(2,1,1)a, 8(2)a, 8(2,1,1), 8(3).
	\end{center}
\end{prop}

\begin{proof}
Each of these thirteen graphs is realized over $\bbQ(\omega)$, as indicated in Appendix~\ref{app:all_data} (and Figure~\ref{fig:directed_system}). From \cite[\textsection 5]{doyle/faber/krumm:2014}, the only graphs $G(f_c,K)$ with $c \in K$ that are not strongly admissible are 3(2), 4(1), 5(1,1)a, and 7(2,1,1)a. Just as in the proof of Proposition~\ref{prop:main_i}, we need only consider those $c \in K$ such that $G(f_c,K)$ is strongly admissible with cycle structure (1,1), (2), (3), or (1,1,2). It suffices to show that for such $c \in K$ the graph $G(f_c,K)$ is isomorphic to one of the following:
	\begin{center}
	0, 4(1,1), 4(2), 6(1,1), 6(2), 6(3), 8(2)a, 8(2,1,1), 8(3).
	\end{center}
By considering Figure~\ref{fig:directed_system}, it remains to show the following:
	\begin{enumerate}
	\item The graph $G(f_c,K)$ does not contain 8(1,1)a, 8(1,1)b, 8(2)b, 10(2,1,1)a, or 10(2,1,1)b.
	\item If $G(f_c,K)$ contains a subgraph isomorphic to 8(3), then $c = -29/16$ and $G(f_c,K) \cong {\rm8(3)}$.
	\end{enumerate}
Part (A) follows from Proposition~\ref{prop:gen1}, since $X_1(G)$ has genus $1$ and $\rk X_1(G)(K) = 0$ for each of the graphs $G$ appearing in (A). Part (B) follows just as in the proof of Proposition~\ref{prop:main_i}, since the Jacobian of the curve associated to 8(3) also has rank $1$ over $K$.
\end{proof}

\appendix
\section{Determining the $\bbQ(i)$-rational points on $X_0(5)$}\label{app:X0(5)}
This appendix contains the calculations that prove that $X_0(5)(\bbQ(i)) = X_0(5)(\bbQ)$, as stated in the proof of Theorem~\ref{thm:5cycle_cyclotomic}. Our proof uses a variant of the usual Chabauty-Coleman machinery; we are grateful to Joseph Wetherell for describing this modification to us. We begin in \textsection \ref{sec:Chab} with a brief summary of the standard Chabauty-Coleman method, and in \textsection \ref{sec:ChabMod} we describe the modified version that we eventually apply in \textsection \ref{sec:ChabCalc} to the curve $X_0(5)$.

\subsection{An overview of the Chabauty-Coleman method}\label{sec:Chab}
We provide a brief description of the method of Chabauty and Coleman here. For further details, we refer the reader to \cite{wetherell:1997} and \cite{mccallum/poonen:2013}, as both provide an excellent introduction to the method.

Let $X$ be a smooth projective curve of genus $g \ge 1$ defined over a number field $K$, and let $J$ be its Jacobian. To ease the exposition, we assume that $X$ has a $K$-rational point $P_0$. We further make the crucial assumption that $r := \rk J(K) < g$. Let $\frakp \in \Spec \OK$ be a prime of good reduction for $X$, let $K_\frakp$ be the $\frakp$-adic completion of $K$, let $\calO_\frakp \subset K_\frakp$ be the ring of integers, and let $k_\frakp$ be the residue field of $K_\frakp$. The \textbf{residue disk} of a point $P \in X(K_\frakp)$ is the set
	\[
		\calU(P) := \{Q \in X(K_\frakp) : \tilde{Q} = \tilde{P}\},
	\]
where the tilde denotes reduction modulo $\frakp$. We let $\{P_0,\ldots,P_{m-1}\}$ be a full set of residue class representatives for $X(K_\frakp)$; that is, $P_0,\ldots,P_{m-1}$ lie in distinct residue disks, and $X(k_\frakp) = \{\tilde{P_0},\ldots,\tilde{P_{m-1}}\}$. Note that the reduction map $X(K_\frakp) \longto X(k_\frakp)$ is surjective by Hensel's lemma, since we have assumed $X$ is smooth over $k_\frakp$.

Since $X$ has genus $g$, the space $H^0(X, \Omega_X^1)$ of regular differentials on $X$ has dimension $g$; let $\omega_1,\ldots,\omega_g$ be a basis. The Albanese embedding
	\begin{align*}
		i_{P_0} : X &\longhookrightarrow J\\
		P &\longmapsto [P - P_0]
	\end{align*}
induces an isomorphism $H^0(X, \Omega_X^1) \cong H^0(J, \Omega_J^1)$, so we freely identify the two spaces.

The set of $K_\frakp$-rational points on $J$ forms a $\frakp$-adic Lie group, and one defines a {\it logarithm} map on $J(K_\frakp)$ by
	\begin{align*}
		\Log : J(K_\frakp) &\longto K_\frakp^g\\
			\calP &\longmapsto \left(\int_\calO^\calP \omega_1,\ldots,\int_\calO^\calP \omega_g\right),
	\end{align*}
which is locally an analytic isomorphism. Integration is defined in such a way that, for a degree-$0$ divisor $D = \sum_{j=1}^n Q_j - \sum_{j=1}^n Q_j'$, we have
	\[
		\int_\calO^{[D]} \omega = \int_{\sum_{j=1}^n Q_j'}^{\sum_{j=1}^n Q_j} \omega = \sum_{j=1}^n \int_{Q_j'}^{Q_j} \omega.
	\]
The closure $\overline{J(K)}$ of $J(K)$ in $J(K_\frakp)$ has dimension at most $r$, hence the same is true for $\Log(\overline{J(K)})$, which implies that
	\[
		\Ann(J(K)) := \left\{\omega : \int_\calO^\calP \omega = 0 \text{ for all } \calP \in J(K)\right\} \subset H^0(J_{K_\frakp}, \Omega_{J_{K_\frakp}^1)
}	\]
is a $K_\frakp$-vector subspace of dimension at least $g - r$. Since we have assumed that $r < g$, $\Ann(J(K))$ is nomempty; we call elements of $\Ann(J(K))$ {\bf annihilating differentials}.

We now turn to the matter of computing $\frakp$-adic integrals on $J(K_\frakp)$. By Riemann-Roch, every element of $J(K_\frakp)$ may be written in the form
	\[
		\calP = [Q_1 + \cdots + Q_g - gP_0],
	\]
where $Q_1 + \cdots + Q_g$ is a $\Gal(\overline{K_\frakp}/K_\frakp)$-invariant divisor on $X$. Thus,
	\[
		\int_\calO^\calP \omega = \sum_{j=1}^n \int_{P_0}^{Q_j} \omega,
	\]
so we need only determine integrals of the form $\int_{P_0}^Q \omega$ with $Q \in X(K_\frakp)$.

First, suppose $Q$ is in the residue disk $\calU(P_0)$. The disk $\calU(P_0)$ is analytically isomorphic to the maximal ideal of $\calO_\frakp$; abusing notation, we will also call this maximal ideal $\frakp$. Let $t$ be a uniformizing parameter at $P_0$, and write $Q_t$ for the family of points in $\calU(P_0)$ parametrized by $t \in \frakp$. At $P_0$, one expands each differential $\omega$ as a power series $\omega(t)\ dt$ centered at $t = 0$, and integrating yields
	\[
		\int_{P_0}^{Q_t} \omega_j = \lambda_j(t) := \int \omega_j(t)\ dt,
	\]
where the constant of integration is equal to $0$. The power series $\lambda_j(t)$ can be determined up to arbitrary $t$-adic and $\frakp$-adic precision, and it converges on the entire residue disk $\calU(P_0)$.

Now, suppose $Q$ is not in $\calU(P_0)$. Then $Q \in \calU(P_\ell)$ for some $1 \le \ell \le m - 1$, and we can write
	\[
		\int_{P_0}^Q \omega = \int_{P_0}^{P_\ell} \omega + \int_{P_\ell}^Q \omega.
	\]
The second integral on the right hand side can be evaluated using power series as in the previous paragraph, while the first integral generally requires more involved cohomological techniques; see, for example, the articles \cite{coleman:1982, balakrishnan/bradshaw/kedlaya:2010, balakrishnan:2013, balakrishnan:2015, balakrishnan/tuitman}. For this reason, it is preferable (when possible) to avoid computing integrals across distinct residue disks and compute only those integrals within a given residue disk; such integrals are typically called {\it tiny integrals}.

To determine a basis for $\Ann(J(K))$ (up to the desired $\frakp$-adic precision), one first finds elements $\calP_1,\ldots,\calP_r$ that span a finite index\footnote{Ideally, one determines a finite index subgroup $H \subseteq J(K)$ whose index is coprime to $p$, though by increasing the $\frakp$-adic precision of the integral calculations it suffices to have an upper bound for the $p$-adic valuation of the exponent of the quotient $J(K)/H$. For the calculation in \textsection \ref{sec:ChabCalc}, the index of our finite index subgroup will be coprime to $p = 2$, so we do not need to worry about this.} subgroup of $J(K)$. Since $\langle \calP_1,\ldots,\calP_r\rangle$ has finite index in $J(K)$, and since $\calP \longmapsto \int_\calO^\calP \omega$ is linear, we have that $\omega \in \Ann(J(K))$ if and only if
	\[
		\int_\calO^{\calP_j} \omega = 0 \quad\text{for all } 1 \le j \le r.
	\]
Thus, one computes $\int_\calO^{\calP_j} \omega_k$ for each $1 \le j \le r$ and $1 \le k \le g$, then calculates a basis for the kernel of the matrix whose $(j,k)$-entry is $\int_\calO^{\calP_j} \omega_k$. Applying a change of coordinates if necessary, we may suppose that $\{\omega_1,\ldots,\omega_s\}$ forms a basis for $\Ann(J(K))$; note that $s \ge g - r$.

Finally, for each residue disk $\calU(P_j)$ on $X(K_\frakp)$, one uses standard $\frakp$-adic techniques to bound the number of elements $Q \in \calU(P_j)$ such that $\int_{P_0}^Q \omega = 0$ for all $\omega \in \Ann(J(K))$. More precisely, for each $1 \le k \le s$ one can write $\int_{P_0}^Q\omega_k$ as a power series locally at $\calP_j$, then use Stra\ss mann's theorem to give an upper bound for the number of zeroes of each power series. One then hopes that information of the Stra\ss mann bounds is enough to conclude that the known $K$-rational points in the given residue disk are the only such points.

\subsection{A modification of the usual technique}\label{sec:ChabMod}
For the curve $X_0(5)$ over the field $\bbQ(i)$, we cannot use the Chabauty-Coleman technique exactly as described in the previous section because $J_0(5)(\bbQ(i))$ has rank $2$, equal to the genus of $X_0(5)$. However, a heuristic that Siksek \cite{siksek:2013} attributes to Wetherell---and which Wetherell has described to the author---suggests that a Chabauty-type method should typically give a bound on the number of $K$-rational points on a curve $X$ of genus $g$ under the weaker hypothesis that $J(K)$ has rank $r \le d(g - 1)$, where $d := [K:\bbQ]$. Though our calculations do not require explicitly working with them, the heuristic involves looking at the Weil restrictions $V := \Res_{K/\bbQ} X$ and $A := \Res_{K/\bbQ} J$, which are a $d$-dimensional variety and a $dg$-dimensional abelian variety, respectively. The key property that we take advantage of is that $K$-rational points on $X$ and $J$ correspond to $\bbQ$-rational points on $V$ and $A$, respectively. Moreover, $J(K)$ and $A(\bbQ)$ are isomorphic as groups, so in particular they have the same rank, hence the closure of $A(\bbQ)$ in $A(\bbQ_p)$ is still at most $r$-dimensional.

For concreteness (and because this is the case we ultimately require), we now fix the field $K = \bbQ(i)$, and we let $p$ be a prime such that $-1$ is not a square in $\bbQ_p$, so that $K_\frakp := \bbQ_p(i)$ is a nontrivial extension. In the standard Chabauty-Coleman method, one has a $g$-dimensional space of linear maps $J(\QQbar_p) \longto \QQbar_p$ given by integration of the regular differentials; using the Weil restriction $A$, we get a $2g$-dimensional space of linear maps $J(\bbQ_p(i)) \longto \bbQ_p$ by taking the ``real" and ``imaginary" parts of those integrals. In other words, if $\{\omega_1,\ldots,\omega_g\}$ is a basis for $H^1(J, \Omega_J^1)$, then
	\begin{equation}\label{eq:ReImIntegrals}
		\Re\left(\int_\calO^\bullet \omega_1\right),\ \Im\left(\int_\calO^\bullet \omega_1\right),\ \ldots,\ \Re\left(\int_\calO^\bullet \omega_g\right),\ \Im\left(\int_\calO^\bullet \omega_g\right)
	\end{equation}
are $2g$ independent linear maps $J(\bbQ_p(i)) \cong A(\bbQ_p) \longto \bbQ_p$. Here we abuse notation and write
	\[
		\Re(a + bi) = a \quad\text{and}\quad \Im(a + bi) = b
	\]
for $a,b \in \bbQ_p$, recalling that we have assumed $i \notin \bbQ_p$. Each of these linear maps can be written locally as a power series, but now in {\it two} variables, corresponding to the real and imaginary parts of the local parameter. (Alternatively, because the Weil restriction $V$ of $X$ is $2$-dimensional, we require a pair of parameters at a given point on $V$.)

If we have $r \le 2g - 2$, then also $\dim \overline{A(\bbQ)} \le r \le 2g - 2$, so the space $\Ann(A(\bbQ))$ of those maps \eqref{eq:ReImIntegrals} that vanish on $A(\bbQ) \cong J(K)$ has dimension at least $2$. Moreover, since $V(\bbQ_p)$ is $2$-dimensional, we expect (just by counting dimensions) that the set
	\[
		\{P \in V(\bbQ_p) \mid \lambda(P) = 0 \text{ for all } \lambda \in \Ann(A(\bbQ))\}
	\]
is finite. This is not always the case (see the remark on \cite[p. 768]{siksek:2013}), but when the set is finite, it remains to find a bound for the number of common zeroes of at least two power series in two variables. 
Under favorable conditions, this can be done using standard $p$-adic techniques, like Hensel's lemma. We now state a multivariate version of Hensel's lemma appearing in notes of Keith Conrad \cite[Theorem 3.8]{conrad:hensel_2020}.
	\begin{lem}\label{lem:hensel}
	Let $k$ be a complete, non-archimedean valued field with valuation $v$ and ring of integers $\calO_k$.
	Let $f_1,\ldots,f_m \in \calO_k[x_1,\ldots,x_n]$ with $m \le n$. Consider the \emph{truncated} Jacobian matrix
		\[
			J_{\bsf, m} := \left(\partial f_i/\partial x_j\right)_{1 \le i,j \le m}.
		\]
	Suppose $\bsa = (a_1,\ldots,a_n) \in \calO_k^n$ satisfies $v(f_j(\bsa)) > 2v(\det J_{\bsf, m}(\bsa))$ for all $1 \le j \le m$.
	Then there exist unique $b_1,\ldots,b_m \in \calO_k$ such that $f_j(b_1,\ldots,b_m,a_{m+1},\ldots,a_n) = 0$ for all $1 \le j \le m$ and $v(a_j - b_j) > v(J_{\bsf, m}(\bsa))$ for all $1 \le j \le m$.
	\end{lem}

\subsection{The calculation}\label{sec:ChabCalc}
Throughout this section, we let $K := \bbQ(i)$.
In Section \ref{sec:X0(5)} we gave the following model for $X_0(5)$, originally calculated in \cite{flynn/poonen/schaefer:1997}:
	\[
		y^2 = f(x) := x^6 + 8x^5 + 22x^4 + 22x^3 + 5x^2 + 6x + 1.
	\]
In this section, we prove the following.
\begin{thm}\label{thm:X0(5)i}
	\[
		X_0(5)(K) = X_0(5)(\bbQ) = \{(0,-1), (0,1), (-3,-1), (-3,1), \infty^+, \infty^-\}.
	\]
\end{thm}

\subsubsection{Generators for $J_0(5)(K)$}\label{sec:J0(5)gens}
A two-descent in \cite[{\tt main.txt}]{doyle:cyclotomic} shows that the Jacobians of both $X_0(5)$ and its twist $X_0(5)^{(-1)}$, given by $y^2 = -f(x)$, have rank $1$ over $\bbQ$. Therefore, $J_0(5)(K)$ has rank $2$ (by Lemma~\ref{lem:rank_twist}), which precludes a standard Chabauty-Coleman procedure to determine $X_0(5)(K)$. We write $J_0(5)$ and $J_0(5)^{(-1)}$ for the Jacobians of $X_0(5)$ and $X_0(5)^{(-1)}$, respectively, and recall that we write $\{P,Q\}$ to represent the point $[P + Q - \infty^+ - \infty^-]$ on the Jacobian. We now give a basis for the Mordell-Weil group $J_0(5)(K)$.

\begin{lem}\label{lem:indep_pts}
Let $\alpha$ and $\alpha'$ be the two roots of $x^2 + 3x + 1/2$, and let $\beta$ and $\beta'$ be the two roots of $x^2 + 3x + (1 - i)$.
Consider the following points on $J_0(5)(K)$:
	\begin{align*}
		\calP_1 &= \{\infty^+, \infty^+\}\\
		\calP_2 &= \{(\alpha,\alpha i/2), (\alpha', \alpha' i/2)\}\\
		\calP_3 &= \{(\beta, -(1 + i)\beta - (2 + i)), (\beta', -(1 + i)\beta' - (2 + i))\}.
	\end{align*}
Then 
	\begin{enumerate}
	\item $2\calP_3 = \calP_1 + \calP_2$, and
	\item $J_0(5)(K) = \langle\calP_1, \calP_3\rangle = \langle\calP_2, \calP_3\rangle$.
	\end{enumerate}
\end{lem}

\begin{proof}[Proof of Lemma~\ref{lem:indep_pts}]
Part (A) is a straightforward calculation (see \cite[\tt KummerSurface.txt]{doyle:cyclotomic}), so we need only prove (B).

A computation in \cite[\tt Chabauty.txt]{doyle:cyclotomic} involving canonical heights on their respective Jacobians shows that $\calP_1 = \{\infty^+,\infty^+\}$ and $\{(\alpha,\alpha/2), (\alpha', \alpha'/2)\}$ are generators for the Mordell-Weil groups $J_0(5)(\bbQ)$ and $J_0(5)^{(-1)}(\bbQ)$, respectively. Note that $\calP_2$ is the point on $J_0(5)(K)$ corresponding to the point $\{(\alpha,\alpha/2), (\alpha', \alpha'/2)\}$ on the twist; more precisely, if $\phi: J_0(5)^{(-1)} \to J_0(5)$ is the isomorphism induced by the map
	\begin{align*}
		X_0(5)^{(-1)} &\longto X_0(5)\\
			(x,y) &\longmapsto (x, iy),
	\end{align*}
then $\calP_2 = \phi\big(\{(\alpha,\alpha/2), (\alpha', \alpha'/2)\}\big)$.

It follows from Lemma~\ref{lem:rank_twist} that the quotient $J_0(5) / \langle \calP_1,\calP_2\rangle$ has exponent at most $2$.
We now observe that $\calP_1$ and $\calP_2$ themselves cannot be doubles of points in $J_0(5)(K)$. Indeed, suppose there exists $\calQ \in J_0(5)(K)$ such that $2\calQ = \calP_1$. Then
	\[
		\tau(2\calQ) = \tau\calP_1 = \calP_1 = 2\calQ,
	\]
so $2(\calQ - \tau\calQ) = 0$. Since $J_0(5)(K)$ is torsion-free, this implies $\calQ = \tau\calQ$, hence $\calQ \in J_0(5)(\bbQ)$, contradicting the fact that $\calP_1$ generates $J_0(5)(\bbQ)$. The same argument shows that we could not have $2\calQ = \calP_2$.

Finally, we show that $\calP_2$ and $\calP_3$ generate $J_0(5)(K)$. It then follows immediately that $\calP_1$ and $\calP_3$ also form a generating set, since $\calP_2 = -\calP_1 + 2\calP_3$. Certainly $\langle \calP_1, \calP_2\rangle \subset \langle \calP_2,\calP_3\rangle$, so let $\calQ \in J_0(5)(K) \setminus \langle \calP_1,\calP_2\rangle$. As argued above, we can write $2\calQ = a\calP_1 + b\calP_2$ for some $a,b \in \bbZ$. We claim that
	\[
		\calQ = \left\lfloor\frac{a}{2}\right\rfloor\calP_1  + \left\lfloor\frac{b}{2}\right\rfloor\calP_2 + \calP_3,
	\]
from which it follows (since $\calP_1 \in \langle \calP_2,\calP_3\rangle$) that $\calQ \in \langle\calP_2,\calP_3\rangle$.
Set $\calQ' := \calQ - \left\lfloor\frac{a}{2}\right\rfloor\calP_1  - \left\lfloor\frac{b}{2}\right\rfloor\calP_2$; thus, our claim is that $\calQ' = \calP_3$.
By construction, we have
	\[
		2\calQ' \in \{\calO, \calP_1,\calP_2,\calP_1 + \calP_2\}.
	\]
Since $\calQ$ is not in $\langle\calP_1,\calP_2\rangle$, and since $J_0(5)(K)$ is torsion-free, we have $2\calQ' \ne \calO$; further, since $\calP_1$ and $\calP_2$ are not doubles in $J_0(5)(K)$, it must be that $2\calQ' = \calP_1 + \calP_2 = 2\calP_3$. Once again applying the fact that $J_0(5)(K)$ is torsion-free, we conclude that $\calQ' = \calP_3$, as claimed.
\end{proof}

\subsubsection{A model with good reduction at $2$}
Let $\frakp := (1 + i)$ be the unique prime ideal in $\OK$ lying over $2$, and let $K_\frakp := \bbQ_2(i)$ be the completion of $K$ at $\frakp$. We will abuse notation and let $\frakp$ also denote the maximal ideal in the ring of integers $\calO_\frakp$ of $K_\frakp$.

The model of $X_0(5)$ given by $y^2 = f(x)$ above has bad reduction at the prime $\frakp$, but there exists a model with good reduction at $\frakp$. Let
	\begin{align*}
		g(x) &:= 2x^5 + 5x^4 + 5x^3 + x^2 + x,\\
		h(x) &:= -(x^3 + x + 1),
	\end{align*}
let $X$ be the curve defined by $y^2 + h(x)y = g(x)$, and let $J$ be the Jacobian of $X$. Then $X$ has good reduction at $\frakp$, and we have an isomorphism
	\begin{align}
		X_0(5) &\longto X \notag\\
		(x,y) &\longmapsto \left(x, \frac{1}{2}(y - h(x))\right) \label{eq:X0(5)isom}\\
		(x, 2y + h(x)) &\longmapsfrom (x,y).\notag
	\end{align}
Theorem~\ref{thm:X0(5)i} is therefore equivalent to the following:
\begin{thm}\label{thm:X05_change_of_variables}
	\[
		X(K) = X(\bbQ) = \{(0,0), (0,1), (-3,-15), (-3,-14), \infty^+, \infty^-\}.
	\] 
\end{thm}
For convenience, we will let $\bsP_0, \ldots, \bsP_5$ be the six rational points on $X(K)$, ordered as in the statement of Theorem~\ref{thm:X05_change_of_variables}. In particular, $\bsP_0 = (0,0)$ will be our chosen base point.

Since $K/\bbQ$ is ramified over $2$, the local field $K_\frakp$ is (totally) ramified over $\bbQ_2$, thus the residue field of $K_\frakp$ is still $\bbF_2$. A simple calculation shows that
	\[
		X(\bbF_2) = \{(0,0), (0,1), (1,1), (1,0), \infty^+,\infty^-\},
	\]
so each mod-$2$ residue class on $X$ contains exactly one $\bbQ$-rational point. Thus, it suffices to show that each of these six residue classes contains exactly one $K$-rational point.

Applying the isomorphism $J_0(5) \to J$ induced by \eqref{eq:X0(5)isom}, we get the following version of Lemma~\ref{lem:indep_pts} on $J(K)$:

\begin{cor}\label{cor:indep_pts}
Let $\alpha$ and $\alpha'$ be the two roots of $x^2 + 3x + 1/2$, and let $\beta$ and $\beta'$ be the two roots of $x^2 + 2x + (1 - i)$.
Consider the following points on $J(K)$:
	\begin{align*}
	\calD_1 &:= \{\infty^+, \infty^+\},\\
	\calD_2 &:= \left\{\left(\alpha, \frac{(i + 19)\alpha + 5}{4}\right), \left(\alpha', \frac{(i + 19)\alpha' + 5}{4}\right)\right\},\\
	\calD_3 &:= \left\{\left(\beta, 4\beta + (1 - 2i)\right), \left(\beta', 4\beta' + (1 - 2i)\right)\right\}.
	\end{align*}
Then 
	\begin{enumerate}
	\item $2\calD_3 = \calD_1 + \calD_2$, and
	\item $J(K) = \langle\calD_1, \calD_3\rangle = \langle\calD_2, \calD_3\rangle$.
	\end{enumerate}
\end{cor}

\subsubsection{The kernel of reduction}

We will ultimately need to find two independent linear functions on $J(\bbQ_2(i))$ (linear combinations of those appearing in \eqref{eq:ReImIntegrals}) that vanish on $J(K)$. For a point $\calP = [\sum Q_i - \sum P_i]$ on $J$, it is convenient to integrate from $\calO$ to $\calP$ by writing
	\[
		\int_\calO^\calP \omega = \int_\calO^{\sum [Q_i - P_i]} \omega = \sum \int_{P_i}^{Q_i} \omega;
	\]
as mentioned in \textsection \ref{sec:Chab}, such integrals are straightforward to compute using power series when $P_i$ and $Q_i$ are in the same residue disk. Since we have chosen our base point to be $\bsP_0 = (0, 0)$, and since every element of $J(\QQbar)$ may be represented in the form $[P + Q - 2\bsP_0]$ by Riemann-Roch, we seek to describe the points $\calD = [P + Q - 2\bsP_0]$ on $J(K)$ such that $P$ and $Q$ lie in the same residue disk as $\bsP_0$. Since $\bsP_0$ does not reduce to a Weierstrass point over $\bbF_2$, this is equivalent to $\calD$ being in the kernel---which we denote $J^0(K)$---of the mod-$\frakp$ reduction map.

We now write each of $\calD_1$, $\calD_2$, and $\calD_3$ in the form $\calD_j = [P_j' + Q_j' - 2\bsP_0]$. If we let
	\begin{align*}
		g_1'(x) &= x^2 + 4x + \frac{1}{3},\\
		g_2'(x) &= x^2 + \frac{742 - 44i}{325}x - \frac{512 + 66i}{325},\text{ and}\\
		g_3'(x) &= x^2 + \frac{13}{5}x + \frac{3 - 7i}{5},
	\end{align*}
then for each $j \in \{1,2,3\}$, the $x$-coordinates of $P_j'$ and $Q_j'$ are the roots of $g_j'$.

For $j \in \{1,3\}$, at least one of the roots of $g_j'$ reduces to $1$ in $\bbF_2$; in particular, at least one of $P_j'$ or $Q_j'$ does not reduce to $\bsP_0$ modulo $\frakp$, so neither $\calD_1$ nor $\calD_3$ is in the kernel of reduction. On the other hand, both roots of $g_2'$ vanish modulo $\frakp$; moreover, the $y$-coordinates of $P_2'$ and $Q_2'$ are determined from the $x$-coordinates by
	\[
		y = \frac{27181 - 692i}{8125}x - \frac{22116 + 3638i}{8125},
	\]
so both of the $y$-coordinates also vanish modulo $\frakp$. Therefore, $\calD_2$ is in the kernel of reduction.

Since $|J(\bbF_2)| = 19$, it follows that $19\calD_1$, $\calD_2$, and $19\calD_3$ are all elements of $J^0(K)$. We set
	\begin{align*}
	\calE_1 &:= \calD_2,\\
	\calE_2 &:= 19\calD_3 - 9\calD_2.
	\end{align*}
(We remark that we chose $\calE_2 = 19\calD_3 - 9\calD_2$ rather than just $19\calD_3$ essentially to make the corresponding expression in \eqref{eq:Emin_polys} much more compact.) Then $J^0(K)$ is generated by $\calE_1$ and $\calE_2$, and $[J(K) : J^0(K)] = |J(\bbF_2)| = 19$. Finally, for each $j \in \{1,2\}$, we write $\calE_j = [P_j + Q_j - 2\bsP_0]$, where the $x$-coordinates of $P_j$ and $Q_j$ are the roots of the following polynomials:
	\begin{equation}\label{eq:Emin_polys}
		\begin{split}
			\calE_1:&\ x^2 + \frac{742 - 44i}{325}x - \frac{-512 + 66i}{325}\\
			\\
			\calE_2:&\ x^2 + \frac{7801337 - 1823949i}{2442505}x + \frac{948975 + 120593i}{488501}.
		\end{split}
	\end{equation}

\subsubsection{Annihilating differentials for $J(K)$}
We now determine a space of linear maps $J(\bbQ_2(i)) \longto \bbQ_2$ that vanish on $J(K)$. It suffices to find maps that vanish on the finite index subgroup $J^0(K) \subset J(K)$, so we want linear maps that vanish at $\calE_1$ and $\calE_2$.

For points $P$ in the residue class of $\bsP_0$, we can write the integral $P \longmapsto \int_{\bsP_0}^P \omega_j$ as a power series in a uniformizing parameter $t$ at $\bsP_0$. Since $\bsP_0$ has $x$-coordinate $0$ and is not a Weierstrass point, we take $t = x$.

The space of linear differentials on $X$ is spanned by
	\[
		\omega_1 := \frac{dx}{2y + h(x)} = \frac{dx}{-\sqrt{f(x)}} \quad\text{and}\quad \omega_2 := x\omega_1,
	\]
where we take the negative square root because at $\bsP_0$ we have $x = y = 0$, $f(0) = 1$, and $h(0) = -1$. We note that $\omega_1$ and $\omega_2$ are obtained from the standard basis $\{dx/y, xdx/y\}$ on $X_0(5)$ via the isomorphism \eqref{eq:X0(5)isom}.

We denote by $\Lambda_1$ and $\Lambda_2$ the maps $J(\QQbar_2) \longto \QQbar_2$ given by
	\[
		\Lambda_j(\calP) = \int_\calO^\calP \omega_j.
	\]
We now represent $\Lambda_1$ and $\Lambda_2$ locally near $\bsP_0$ as power series in $t = x$ by first determining the Taylor series of $\omega_1$ and $\omega_2$ centered at $t = 0$; this calculation was done in Sage \cite{sage_cyclotomicquad} and appears in \cite[\tt Chabauty.txt]{doyle:cyclotomic}:
	\begin{align*}
		\omega_1(t)
			&=\ \scriptstyle
			    -1 +
			    3t 
			    -11t^2 +
			    56t^3 
			    -283t^4 +
			    1438t^5 
			    -7506t^6 +
			    39723t^7 
			    -211939t^8 +
			    1139043t^9 
			    -6157964t^{10} +
			    33448053t^{11} 
			    -182389282t^{12} \\
			    &\quad\quad \scriptstyle +
			    997848854t^{13} 
			    -5474673325t^{14} +
			    30110184065t^{15}
			    -165957302527t^{16} +
			    916424740644t^{17}
			    -5069007570927t^{18} \\
			    &\quad\quad \scriptstyle +
			    28080034612882t^{19}
			    -155759823221656t^{20} +
			    865048247560705t^{21}
			    -4809544720320519t^{22} +
			    26767288658743629t^{23}\\
			    &\quad\quad \scriptstyle
			    -149109354289320238t^{24} +
			    831329586241569831t^{25}
			    -4638535883774463494t^{26} +
			    25900170663332468144t^{27}\\
			    &\quad\quad \scriptstyle
			    -144715739340500871241t^{28} +
			    809096110462736894221t^{29}
			    -4526238826848117522585t^{30}\\ 
			    &\quad\quad {\scriptstyle +
			    25334445278892249580026t^{31}} + O(t^{32}), \\
		\omega_2(t)
			&=\ \scriptstyle 
			    -t +
			    3t^2 
			    -11t^3 +
			    56t^4 
			    -283t^5 +
			    1438t^6 
			    -7506t^7 +
			    39723t^8 
			    -211939t^9 +
			    1139043t^{10}
			    -6157964t^{11} +
			    33448053t^{12}\\
			    &\quad\quad \scriptstyle
			    -182389282t^{13} +
			    997848854t^{14}
			    -5474673325t^{15} +
			    30110184065t^{16} 
			    -165957302527t^{17} +
			    916424740644t^{18} \\
			    &\quad\quad \scriptstyle
			    -5069007570927t^{19} +
			    28080034612882t^{20} 
			    -155759823221656t^{21} +
			    865048247560705t^{22} 
			    -4809544720320519t^{23}\\
			    &\quad\quad \scriptstyle +
			    26767288658743629t^{24} 
			    -149109354289320238t^{25} +
			    831329586241569831t^{26} 
			    -4638535883774463494t^{27} \\
			    &\quad\quad \scriptstyle +
			    25900170663332468144t^{28} 
			    -144715739340500871241t^{29} +
			    809096110462736894221t^{30} \\
			    &\quad\quad {\scriptstyle
			    -4526238826848117522585t^{31}} + O(t^{32}),
	\end{align*}
where $O(t^n)$ indicates that the remaining terms are divisible by $t^n$. Integrating, we have
	\begin{align*}
	\lambda_1(t) := \int_0^t \omega_1 &= -t + \tfrac{3}{2}t^2 - 
	\cdots - \tfrac{4526238826848117522585}{31}t^{31} +
	    \tfrac{12667222639446124790013}{16}t^{32} + O(t^{33}),\\
	\lambda_2(t) := \int_0^t \omega_2 &= -\tfrac{1}{2}t^2 + t^3 -
	\cdots + \tfrac{809096110462736894221}{31}t^{31}
	    -\tfrac{4526238826848117522585}{32}t^{32} + O(t^{33}).
	\end{align*}
	
We now explain why we require the degree-$32$ approximations to $\lambda_1$ and $\lambda_2$. First, as we will see later, it will suffice for us to perform all of our calculations modulo $2^6$. Now, if $P \in X(\Kbar_\frakp)$ is in the same residue disk as $\bsP_0$, then $v(t(P)) \ge 1$, where $v$ is the valuation on $K_\frakp(P)$ normalized so that $v(K_\frakp(P)^\times) = \bbZ$. Because we are only evaluating the integrals at elements $\calD \in J(K_\frakp)$, we are only considering divisor classes $[P + Q - 2\bsP_0]$ with $P$ and $Q$ at worst {\it quadratic} over $K_\frakp$, hence at worst {\it quartic} over $\bbQ_2$. Thus the valuation $v$ defined on $K_\frakp(P)$ will satisfy $v(2) \le 4$. Because we will ultimately be working modulo $2^6$, we may discard those terms of the power series which are guaranteed to vanish modulo $2^6$ on the residue disk $\calU(\bsP_0)$, which means discarding those terms with $v$-adic valuation at least $24$.

Since $\omega_1$ and $\omega_2$ have $\bbZ_2$-coefficients, the degree $k$ term of $\lambda_j(t(P))$ has valuation at least
	\[
		kv(t(P)) - v(k) \ge k - v(k) \ge k - 4\ord_2(k) \ge k - 4\log_2(k),
	\]
where $\ord_2$ indicates the $2$-adic valuation on $\bbQ$ with $\ord_2(2) = 1$.
The fact that $\lambda_1$ and $\lambda_2$ converge on the residue disk $\calU(\bsP_0)$ follows from the fact that this quantity increases with $k$; moreover, it is a calculus exercise to show that $k - 4\log_2(k) \ge 24$ for all $k \ge 47$, and then an explicit calculation shows that $k - 4\ord_2(k) \ge 24$ for all $33 \le k \le 46$. Since $32 - 4\ord_2(32) = 12 < 24$, we require precisely the degree $32$ approximation of our power series.

Now, we need to evaluate the integrals of $\omega_1$ and $\omega_2$ on the kernel of reduction $J^0(K)$. Since the points $\calE_j = [P_j + Q_j - 2\bsP_0]$ with $j \in \{1,2\}$ generate $J^0(K)$, it suffices to evaluate the integrals at those two points. By construction, each $P_j$ and $Q_j$ is in the same residue class as $\bsP_0$, so we evaluate using the power series $\lambda_1$ and $\lambda_2$ as follows: Writing
	\[
		\lambda_\ell(t) = \sum_{k=0}^\infty \lambda_{\ell,k}t^k,
	\]
for $\ell \in \{1,2\}$, we have
	\begin{align*}
		\Lambda_\ell(\calE_j)
			&= \int_\calO^{\calE_j} \omega_\ell\\
			&= \int_{\bsP_0}^{P_j} \omega_\ell + \int_{\bsP_0}^{Q_j} \omega_\ell\\
			&= \lambda_\ell(t(P_j)) + \lambda_\ell(t(Q_j))\\
			&= \sum_{k=0}^\infty \lambda_{\ell,k}\left(t(P_j)^k + t(Q_j)^k\right).
	\end{align*}
Since $t = x$ on the residue disk containing $\bsP_0$, the coefficients in the last sum are symmetric in the $x$-coordinates of $P_j$ and $Q_j$, so they can be written explicitly in terms of the coefficients of the corresponding polynomials in \eqref{eq:Emin_polys}. Computing in this way, we find the following values:
	\begin{alignat*}{3}
	\Lambda_1(\calE_1) &\equiv 30i &&\mod 2^6\\
	\Lambda_2(\calE_1) &\equiv 42i &&\mod 2^6\\
	\Lambda_1(\calE_2) &\equiv 17 + 47i &&\mod 2^6\\
	\Lambda_2(\calE_2) &\equiv 50 + 53i &&\mod 2^6
	\end{alignat*}
Thus, separating into ``real" and ``imaginary" parts we have
	\begin{alignat*}{8}
	\Re\Lambda_1(\calE_1) & \equiv 0 && \mod 2^6 &\qquad\qquad&
	\Re\Lambda_1(\calE_2) &\equiv 17 &&\mod 2^6\\
	\Im\Lambda_1(\calE_1) &\equiv 30 &&\mod 2^6
	& &	\Im\Lambda_1(\calE_2) &\equiv 47 &&\mod 2^6\\
	\Re\Lambda_2(\calE_1) &\equiv 0 &&\mod 2^6
	& & \Re\Lambda_2(\calE_2) &\equiv 50 &&\mod 2^6\\
	\Im\Lambda_2(\calE_1) &\equiv 42 &&\mod 2^6
	& & \Im\Lambda_2(\calE_2) &\equiv 53 &&\mod 2^6.
	\end{alignat*}
	
Thus, if $\Mu = c_1\Re\Lambda_1 + c_2\Im\Lambda_1 + c_3\Re\Lambda_2 + c_4\Im\Lambda_2$ is a map that vanishes at both $\calE_1$ and $\calE_2$, with $\bbZ_2$-coefficients $c_1$, $c_2$, $c_3$, and $c_4$, then we have
	\begin{align*}
		30c_2 + 42c_4 &\equiv 0 \mod 2^6\\
		17c_1 + 47c_2 + 50c_3 + 53c_4 &\equiv 0 \mod 2^6.
	\end{align*}
Since both coefficients of the first congruence are (exactly) divisible by $2$, one naturally obtains congruences modulo $2^5$:
	\begin{equation}\label{eq:cong}
		\begin{split}
		15c_2 + 21c_4 &\equiv 0 \mod 2^5\\
		17c_1 + 15c_2 + 18c_3 + 21c_4 &\equiv 0 \mod 2^5.
		\end{split}
	\end{equation}
A bit of linear algebra shows that the space of solutions (mod $2^5$) to \eqref{eq:cong} is spanned by
	\[
		(c_1, c_2, c_3, c_4) \in \{(2, 0, 7, 0), (0, 1, 0, 13)\}.
	\]
Moreover, the Jacobian matrix associated to the system \eqref{eq:cong} is
	\[
		\begin{pmatrix}
		0 & 15 & 0 & 21\\
		17 & 15 & 18 & 21,
		\end{pmatrix}
	\]
and the matrix obtained by taking just the first two columns has determinant with valuation $0$, so Hensel's lemma (Lemma~\ref{lem:hensel}) implies that each solution modulo $2^5$ lifts to a solution $(c_1,c_2,c_3,c_4) \in \bbZ_2^4$ to the system
	\begin{align*}
	c_1\Re\Lambda_1(\calE_1) + c_2\Im\Lambda_1(\calE_1) + c_3\Re\Lambda_2(\calE_1) + c_4\Im\Lambda_2(\calE_1) &= 0\\
	c_1\Re\Lambda_1(\calE_2) + c_2\Im\Lambda_1(\calE_2) + c_3\Re\Lambda_2(\calE_2) + c_4\Im\Lambda_2(\calE_2) &= 0.
	\end{align*}

Since we have two solutions which are independent modulo $2^5$, their lifts are linearly independent as well. Therefore, the space of linear maps vanishing at both $\calE_1$ and $\calE_2$---hence vanishing on all of $J(K)$---is spanned by maps $\Mu_1$ and $\Mu_2$ satisfying
	\begin{alignat*}{3}
		\Mu_1 &\equiv 2\Re\Lambda_1 + 7\Re\Lambda_2 &&\mod 2^5\\
		\Mu_2 &\equiv \Im\Lambda_1 + 13\Im\Lambda_2 &&\mod 2^5.
	\end{alignat*}
Note that we are now using the fact that the index $[J(K) : J^0(K)] = 19$ is coprime to $2$, since this implies that a linear map vanishes modulo $2^n$ on $J^0(K) = \langle \calE_1, \calE_2\rangle$ if and only if it vanishes modulo $2^n$ on all of $J(K)$.

\subsubsection{Residue class computations}
As mentioned immediately following Theorem~\ref{thm:X05_change_of_variables}, it suffices to show that each rational point $\bsP_k$ ($0 \le k \le 5$) is the unique element of $X(K)$ in its residue class. For any $P \in X(K)$, we will have
	\begin{align*}
		\Mu_1([P - \bsP_0]) &= \left(2 + O(2^5)\right)\Re\int_{\bsP_0}^P \omega_1 + \left(7 + O(2^5)\right)\Re\int_{\bsP_0}^P \omega_2 = 0,\\
		\Mu_2([P - \bsP_0]) &= \left(1 + O(2^5)\right)\Im\int_{\bsP_0}^P \omega_1 + \left(13 + O(2^5)\right)\Im\int_{\bsP_0}^P \omega_2 = 0.
	\end{align*}
As discussed previously, it is convenient to work with tiny integrals so that we need only evaluate integrals using power series on residue disks. This is straightforward when $P$ is in the same residue disk as $\bsP_0$, while it is generally more difficult to integrate across different residue disks. Fortunately, in our case, every residue disk contains a rational point, so for any $P \in X(K_\frakp)$ there exists a rational point $\bsP_k$ in the same residue disk as $P$. Using linearity of $\Mu_1$ and $\Mu_2$ yields
	\[
		\Mu_\ell([P - \bsP_0]) = \Mu_\ell([P - \bsP_k]) + \Mu_\ell([\bsP_k - \bsP_0]) = \Mu_\ell([P - \bsP_k])
	\]
for $\ell \in \{1,2\}$, where the final equality follows from the fact that $[\bsP_k - \bsP_0] \in J(K)$ and that $\Mu_1$ and $\Mu_2$ vanish on $J(K)$. Therefore, for each $0 \le k \le 5$, it suffices to bound the number of points $P \in \calU(\bsP_k)$ such that $\Mu_1$ and $\Mu_2$ both vanish at $[P - \bsP_k]$.
We do this explicitly for the residue disks containing $\bsP_0$ and $\bsP_1$. For the remaining disks, the calculations are essentially the same and are therefore omitted; the file {\tt Chabauty.txt}, uploaded as an ancillary file with this article's arXiv submission \cite{doyle:cyclotomic}, contains the necessary Magma computations.\\

\noindent\fbox{$\boldsymbol{\bsP_0 = (0, 0)}$}
We have already computed $\Lambda_1$ and $\Lambda_2$ as power series $\lambda_j(t)$ on the residue disk of $X(K)$ containing $\bsP_0$. Note that we only want zeroes $t = t_0$ which vanish modulo $\frakp$; thus, to separate the ``real" and ``imaginary" parts, we write $t = (1 + i)(T + Ui)$, where $T$ and $U$ are variables now taking values in $\bbZ_2$. Evaluating $\lambda_1$ and $\lambda_2$ at $t = (1 + i)(T + Ui)$ gives two bivariate power series, which we also write as $\lambda_1$ and $\lambda_2$, and which satisfy (modulo $2^5$)
	\begin{align*}
	\lambda_1(T,U)
		&\equiv (\scriptstyle
		16T^{16} + 16T^{12} + 16T^9 + 16T^8U^4 + 16T^8U + 22T^8 + 16T^7 +
		    24T^6U^2 + 16T^6U + 16T^5U^2 + 16T^5U + 28T^5 +
		    16T^4U^8\\
			&\quad\quad\scriptstyle + 4T^4U^4 + 16T^4U^3 + 20T^4U + 8T^4 +
		    16T^3U^4 + 8T^3U^2 + 18T^3 + 24T^2U^6 + 16T^2U^5 +
		    24T^2U^3 + 16T^2U^2 + 22T^2U\\
		    &\quad\quad\scriptstyle + 16TU^8 + 16TU^6 +
		    16TU^5 + 12TU^4 + 10TU^2 + 26TU + 31T + 16U^{16} +
		    16U^{12} + 16U^9 + 22U^8 + 16U^7 + 4U^5\\
		    &\quad\quad{\scriptstyle + 8U^4 + 14U^3 + U})\\
		    &\phantom{\equiv}\ {\scriptstyle + 
		i}(\scriptstyle16T^{10} + 16T^9 + 16T^8U^2 + 16T^8U + 16T^7U + 16T^7 +
		    16T^6U + 24T^6 + 16T^5U^3 + 16T^5U^2 + 28T^5 + 16T^4U^3
		    + 24T^4U^2\\
		    &\quad\quad\scriptstyle + 12T^4U + 16T^3U^5 + 16T^3U^4 + 8T^3U^2 +
		    14T^3 + 16T^2U^8 + 16T^2U^5 + 8T^2U^4 + 8T^2U^3 +
		    22T^2U + 3T^2 + 16TU^8\\
		    &\quad\quad{\scriptstyle + 16TU^7 + 16TU^6 + 12TU^4 +
		    22TU^2 + 31T + 16U^{10} + 16U^9 + 16U^7 + 8U^6 + 28U^5 +
		    14U^3 + 29U^2 + 31U}),\\
	\lambda_2(T,U)
		&\equiv (\scriptstyle
		16T^{16} + 16T^9 + 16T^8U + 28T^8 + 16T^7 + 16T^6U^2 + 16T^6U
		    + 16T^5U^2 + 8T^5U + 8T^4U^4 + 16T^4U^3 + 11T^4 +
		    16T^3U^4\\
		    &\quad\quad\scriptstyle + 16T^3U^3 + 30T^3 + 16T^2U^6 + 16T^2U^5 +
		    30T^2U^2 + 26T^2U + 16TU^8 + 16TU^6 + 8TU^5 + 6TU^2
		    + 2TU + 16U^{16} + 16U^9\\
		    &\quad\quad{\scriptstyle + 28U^8 + 16U^7 + 11U^4 + 2U^3})\\
		    &\phantom{\equiv}\ {\scriptstyle +
		i}(\scriptstyle16T^{10} + 16T^9 + 16T^8U^2 + 16T^8U + 16T^7 + 16T^6U + 4T^6 +
		    16T^5U^2 + 16T^4U^3 + 4T^4U^2 + 16T^3U^4 + 12T^3U +
		    2T^3\\
		    &\quad\quad\scriptstyle + 16T^2U^8 + 16T^2U^5 + 28T^2U^4 + 26T^2U + 31T^2 +
		    16TU^8 + 16TU^6 + 20TU^3 + 26TU^2 + 16U^{10} + 16U^9 +
		    16U^7 + 28U^6\\
		    &\quad\quad{\scriptstyle + 2U^3 + U^2}).
	\end{align*}
Letting $\mu_1$ and $\mu_2$ be the power series representing $\Mu_1$ and $\Mu_2$, respectively, on the residue disk containing $\bsP_0$, we have (again, modulo $2^5$)
	\begin{align*}
	\mu_1(T,U)
		&\equiv \scriptstyle\
		16T^{16} + 16T^9 + 16T^8U + 16T^8 + 16T^7 + 16T^6U + 16T^5U^2
		    + 24T^5U + 24T^5 + 16T^4U^3 + 8T^4U + 29T^4 + 16T^3U^4\\
			&\quad\quad\scriptstyle +
		    16T^3U^3 + 16T^3U^2 + 22T^3 + 16T^2U^5 + 16T^2U^3 +
		    18T^2U^2 + 2T^2U + 16TU^8 + 16TU^6 + 24TU^5 + 24TU^4
		    + 30TU^2\\
		    &\quad\quad\scriptstyle + 2TU + 30T + 16U^16 + 16U^9 + 16U^8 + 16U^7 +
		    8U^5 + 29U^4 + 10U^3 + 2U,\\
	\mu_2(T,U)
		&\equiv \scriptstyle\
		16T^7U + 12T^6 + 16T^5U^3 + 28T^5 + 12T^4U^2 + 12T^4U +
		    16T^3U^5 + 8T^3U^2 + 28T^3U + 8T^3 + 20T^2U^4 +
		    8T^2U^3 + 8T^2U\\
		    &\quad\quad\scriptstyle + 22T^2 + 16TU^7 + 12TU^4 + 4TU^3 +
		    8TU^2 + 31T + 20U^6 + 28U^5 + 8U^3 + 10U^2 + 31U.
	\end{align*}

By construction, $(T, U) = (0, 0)$, which corresponds to the rational point $\bsP_0$, is a $\bbZ_2$-solution to the system $\mu_1(T,U) = \mu_2(T,U) = 0$. We now claim that there is at exactly one additional simultaneous zero $(T_0, U_0) \in \bbZ_2^2$ of $\mu_1$ and $\mu_2$. Before proving this statement, we explain why it then follows that $\bsP_0$ is the only $K$-rational point in its residue disk. Indeed, it was shown by Flynn, Poonen, and Schaefer \cite{flynn/poonen/schaefer:1997} that $\bsP_0$ is the only $\bbQ$-rational point in its residue disk, so if $P$ were another $K$-rational point in the residue disk, it must be a strictly quadratic point. Since $\frakp$ is totally ramified, $\Gal(K/\bbQ)$ preserves $\frakp$, hence the $\Gal(K/\bbQ)$-conjugate of $P$ is a third point in the residue disk of $\bsP_0$, contradicting the fact that there are only two simultaneous zeroes of $\mu_1$ and $\mu_2$. 

Now, let $(T_0, U_0) \in \bbZ_2^2$ be a solution of $\mu_1 = \mu_2 = 0$. By enumerating all solutions to the congruence $\mu_1(T,U) \equiv \mu_2(T,U) \equiv 0 \mod 2^5$, we find that
	\[
		(T_0, U_0) \equiv (8k + 6k', 32 - 8k - 6k') \mod 2^5
	\]
for some integers $0 \le k \le 3$ and $k' \in \{0,1\}$. In particular, a solution must reduce to $(0, 0)$ or $(6, 2)$ modulo $2^3$. It now suffices to show that for each of these two mod-$2^3$ solutions there is a {\it unique} $\bbZ_2$-solution.

Let $\bsJ$ be the Jacobian matrix of $\mu_1$ and $\mu_2$. Then we have
	\begin{align*}
		\bsJ(0,0) &\equiv 
			\begin{pmatrix}
			6 & 2\\
			7 & 7
			\end{pmatrix}
		\mod 2^3,\\
		\bsJ(6, 26) \equiv \bsJ(6,2) &\equiv
			\begin{pmatrix}
			2 & 6\\
			7 & 7
			\end{pmatrix}
		\mod 2^3.
	\end{align*}
Thus $\ord_2\left(\det\bsJ(0,0)\right) = \ord_2\left(\det\bsJ(6, 26)\right) = \ord_2(\pm 28) = 2$. Combining this with the fact that $\ord_2(\mu_\ell(0,0)) \ge 5$ and $\ord_2(\mu_\ell(6,26)) \ge 5$ for $\ell \in \{1,2\}$, the multivariate version of Hensel's lemma (Lemma~\ref{lem:hensel}) implies that there is a unique $\bbZ_2$-solution to $\mu_1 = \mu_2 = 0$ congruent to $(0, 0)$ (resp., $(6, 26)$) modulo $2^3$. On the other hand, we have already checked that any common solution must reduce to $(0,0)$ or $(6, 26)$ modulo $2^3$, so these two must be the only solutions.\\

\noindent\fbox{$\boldsymbol{P_1 = (0,1)}$}
It is now a quick proof to show that $\bsP_1$ is the only $K$-rational point in its residue disk. Let $P$ be such a point. Then the hyperelliptic involution $\iota$ maps $P$ to the residue disk containing $\iota \bsP_1 = \bsP_0$. But we know that $\bsP_0$ is the only $K$-rational point in its residue class, so actually $\iota P = \bsP_0$, and therefore $P = \iota \bsP_0 = \bsP_1$.\\

\noindent\fbox{$\boldsymbol{P_2 = (-3, -15),\ P_3 = (-3,-14),\ P_4 = \infty^+,\ \text{and}\ P_5 = \infty^-}$}
As stated above, the calculations for these four points are nearly identical to those for $\bsP_0$ and $\bsP_1$, so we omit them here. The only difference is that for each of the points $\bsP_2$ and $\bsP_4$ we need to replace the uniformizing parameter $t = x$ with $t = x + 3$ and $t = 1/x$, respectively, and then rewrite the power series $\lambda_1$ and $\lambda_2$ (hence also $\mu_1$ and $\mu_2$) accordingly. The calculations, which appear in \cite[\tt Chabauty.txt]{doyle:cyclotomic}, show that $\bsP_2$ and $\bsP_4$ (hence also $\bsP_3$ and $\bsP_5$) are the unique $K$-rational points in their residue disks, concluding the proof of Theorem~\ref{thm:X05_change_of_variables}, hence also Theorem~\ref{thm:X0(5)i}.

\section{Preperiodic graphs}\label{app:all_data}

We include here all directed graphs that are required for the results of this paper. The graphs are separated into three tables: Table~\ref{tab:graphs_true} contains those graphs that are realized as $G(f_c,K)$ for some cyclotomic quadratic field $K$ and parameter $c \in K$. For each such graph $G$, we also include a representative value of $c$ for which $G(f_c,K) \cong G$, and we include a set $\{P_1,\ldots,P_n\}$ of preperiodic points such that $\PrePer(f_c,K) = \{\pm P_1,\ldots,\pm P_n\}$. A dagger ($\dagger$) indicates that the graph is only realized over $\bbQ(i)$, and a double dagger ($\ddagger$) indicates that it is only realized over $\bbQ(\omega)$. If there is no adornment on a graph $G$, then that graph is realized over both fields, and for the given (rational) representative $c$ we have $G(f_c,\bbQ(i)) = G(f_c,\bbQ(\omega)) \cong G$. Tables~\ref{tab:graphs_false_jxd} and \ref{tab:graphs_false_j} contain graphs that are not realized as $G(f_c,K)$ for a cyclotomic quadratic field $K$ and $c \in K$, but that nonetheless play a role in the proof of the main theorem.

Graphs are labeled by one of two conventions. If a graph was found in the search carried out in \cite{doyle/faber/krumm:2014}, then it is labeled as in that article. Such graphs appear in Tables~\ref{tab:graphs_true} and \ref{tab:graphs_false_jxd}, and their labels are of the form $N(\ell_1,\ell_2,\ldots)$, where $N$ denotes the number of vertices in the graph and $\ell_1,\ell_2,\ldots$ are the lengths of the directed cycles in the graph in nonincreasing order. If more than one isomorphism class of graphs with this data was observed, we add a lowercase Roman letter to distinguish them. The remaining graphs (i.e., those required by our arguments but not found in \cite{doyle/faber/krumm:2014}) appear in Table~\ref{tab:graphs_false_j} and are labeled $G_i$ for $1 \le i \le 6$ as in \cite{doyle:2018quad}.

Finally, as mentioned in Remark~\ref{rem:infinity}, we always omit the connected component of the graph corresponding to the fixed point at infinity.

\begin{table}[h]
\caption{Graphs realized over $\bbQ(i)$ or $\bbQ(\omega)$}
\label{tab:graphs_true}
\begin{tabularx}{\textwidth}{|L|L|L|} \hline
{\large\bf 0} \hfill $c = 2$ & {\large\bf3(2)} \hfill $c = -1$ & {\large\bf4(1)$\ddagger$} \hfill $c = 1/4$
\end{tabularx} \offinterlineskip
\begin{tabularx}{\textwidth}{|H|H|H|}
  & \pic{graph3_2} & \pic{graph4_1}
\end{tabularx} \offinterlineskip
\begin{tabularx}{\textwidth}{|L|L|L|}
$\emptyset$ & $\{0, 1\}$ & $\{1/2, 1/2 + \omega\}$ \\\hline
\end{tabularx} \offinterlineskip
\begin{tabularx}{\textwidth}{|L|L|L|}
{\large\bf4(1,1)} \hfill $c = -6$ & {\large\bf4(2)} \hfill $c = -3$ & {\large\bf5(1,1)a} \hfill $c = -2$
\end{tabularx} \offinterlineskip
\begin{tabularx}{\textwidth}{|H|H|H|}
\pic{graph4_11} & \pic{graph4_2} & \pic{graph5_11a}
\end{tabularx} \offinterlineskip
\begin{tabularx}{\textwidth}{|L|L|L|}
$\{2,3\}$ & $\{1,2\}$ & $\{0,1,2\}$ \\\hline
\end{tabularx} \offinterlineskip
\begin{tabularx}{\textwidth}{|L|L|L|}
{\large\bf5(1,1)b$\dagger$} \hfill $c = 0$&{\large\bf 5(2)a$\dagger$} \hfill $c = i$& {\large\bf6(1,1)} \hfill $c = -10/9$ \\
\end{tabularx} \offinterlineskip
\begin{tabularx}{\textwidth}{|H|H|H|}
\pic{graph5_11b} & \pic{graph5_2a} & \pic{graph6_11}
\end{tabularx} \offinterlineskip
\begin{tabularx}{\textwidth}{|L|L|L|}
$\{0,1,i\}$ & $\{0,i,-1 + i\}$ & $\{2/3,4/3,5/3\}$ \\\hline
\end{tabularx} \offinterlineskip
\begin{tabularx}{\textwidth}{|L|L|L|}
{\large\bf6(2)} \hfill $c = -13/9$ & {\large\bf6(2,1)$\dagger$} \hfill $c = 1/4$& {\large\bf6(3)} \hfill $c = -301/144$
\end{tabularx} \offinterlineskip
\begin{tabularx}{\textwidth}{|H|H|H|}
\pic{graph6_2} & \pic{graph6_21} & \pic{graph6_3}
\end{tabularx} \offinterlineskip
\begin{tabularx}{\textwidth}{|L|L|L|}
$\{1/3,4/3,5/3\}$ & $\{1/2, -1/2 + i, 1/2 + i\}$ & $\{5/12, 19/12, 23/12\}$ \\\hline
\end{tabularx} \offinterlineskip
\begin{tabularx}{\textwidth}{|L|L|L|}
{\large\bf7(2,1,1)a$\ddagger$} \hfill $c = 0$ & {\large\bf8(2)a$\ddagger$} \hfill $c = -5/12$ & {\large\bf8(2,1,1)} \hfill $c = -21/16$
\end{tabularx} \offinterlineskip
\begin{tabularx}{\textwidth}{|H|H|H|}
\pic{graph7_211a} & \pic{graph8_2a} & \pic{graph8_211}
\end{tabularx} \offinterlineskip
\begin{tabularx}{\textwidth}{|L|L|L|}
 & $\{1/6 - 2\omega/3, 5/6 + 2\omega/3,$ & 
\end{tabularx} \offinterlineskip
\begin{tabularx}{\textwidth}{|L|L|L|}
$\{0,1,\omega,1 + \omega\}$ & \hfill $5/6 - \omega/3, 7/6 + \omega/3\}$ & $\{1/4,3/4,5/4,7/4\}$ \\\hline
\end{tabularx} \offinterlineskip
\begin{tabularx}{\textwidth}{|M|M|}
{\large\bf8(3)} \hfill $c = -29/16$ & {\large\bf10(2,1,1)a$\dagger$} \hfill $c = -1/4 + 3i/8$
\end{tabularx} \offinterlineskip
\begin{tabularx}{\textwidth}{|W|W|}
\pic{graph8_3} & \pic{graph10_211a}
\end{tabularx} \offinterlineskip
\begin{tabularx}{\textwidth}{|M|M|}
 & $\{1/4 - i/4, 1/4 + 3i/4, 3/4 + i/4,$ \\
\end{tabularx} \offinterlineskip
\begin{tabularx}{\textwidth}{|M|M|}
$\{1/4,3/4,5/4,7/4\}$ & \hfill $3/4 - 3i/4, 5/4 - i/4\}$ \\\hline
\end{tabularx}\offinterlineskip
\end{table}

\newpage

\begin{table}[h]
\caption{Additional graphs from \cite{doyle/faber/krumm:2014}}
\label{tab:graphs_false_jxd}
\begin{tabularx}{\textwidth}{|L|L|L|}
\hline
{\large\bf8(1,1)a} & {\large\bf8(1,1)b} & {\large\bf8(2)b}
\end{tabularx} \offinterlineskip
\begin{tabularx}{\textwidth}{|H|H|H|}
\pic{graph8_11a} & \pic{graph8_11b} & \pic{graph8_2b} \\ \hline
\end{tabularx} \offinterlineskip
\begin{tabularx}{\textwidth}{|L|L|L|}
{\large\bf8(4)} & {\large\bf10(1,1)a} & {\large\bf10(1,1)b}
\end{tabularx} \offinterlineskip
\begin{tabularx}{\textwidth}{|H|H|H|}
\pic{graph8_4} & \pic{graph10_11a} & \pic{graph10_11b} \\ \hline
\end{tabularx} \offinterlineskip
\begin{tabularx}{\textwidth}{|M|M|}
{\large\bf10(2)} & {\large\bf10(2,1,1)b}
\end{tabularx} \offinterlineskip
\begin{tabularx}{\textwidth}{|W|W|}
\pic{graph10_2} & \pic{graph10_211b} \\ \hline
\end{tabularx} \offinterlineskip
\begin{tabularx}{\textwidth}{|M|M|}
{\large\bf10(3)a} & {\large\bf10(3)b}
\end{tabularx} \offinterlineskip
\begin{tabularx}{\textwidth}{|W|W|}
\pic{graph10_3a} & \pic{graph10_3b} \\ \hline
\end{tabularx} \offinterlineskip
\begin{tabularx}{\textwidth}{|M|M|}
{\large\bf10(3,1,1)} & {\large\bf10(3,2)}
\end{tabularx} \offinterlineskip
\begin{tabularx}{\textwidth}{|W|W|}
\includegraphics[scale=.63]{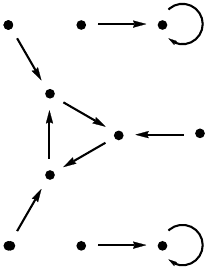} & \pic{graph10_32} \\ \hline
\end{tabularx} \offinterlineskip
\begin{tabularx}{\textwidth}{|M|M|}
{\large\bf12(2,1,1)a} & {\large\bf12(2,1,1)b}
\end{tabularx} \offinterlineskip
\begin{tabularx}{\textwidth}{|W|W|}
\includegraphics[scale=.63]{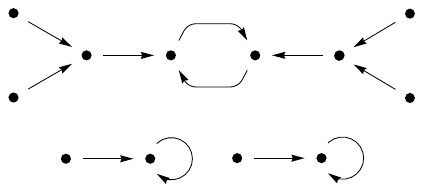} & \pic{graph12_211b} \\ \hline
\end{tabularx} \offinterlineskip
\phantom{Phantom text}
\end{table}
\newpage
\renewcommand{\arraystretch}{1.2}
\begin{table}[h]
\caption{Additional graphs from \cite{doyle:2018quad}}
\label{tab:graphs_false_j}
\begin{tabularx}{\textwidth}{|M|M|}
\hline
{\large$\boldsymbol{G_1}$} & {\large$\boldsymbol{G_2}$}
\end{tabularx} \offinterlineskip
\begin{tabularx}{\textwidth}{|W|W|}
\includegraphics[scale=.5]{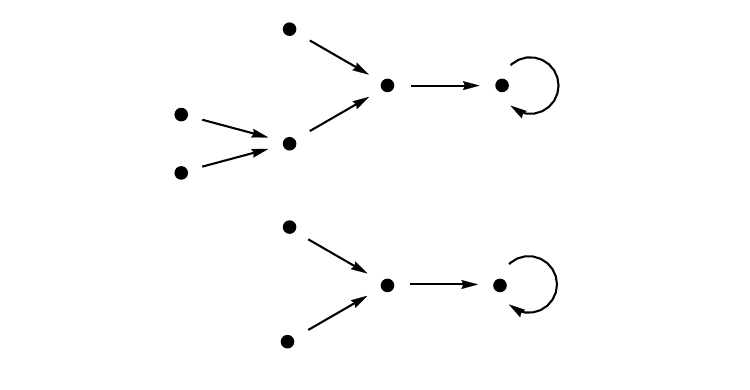} & \includegraphics[scale=.5]{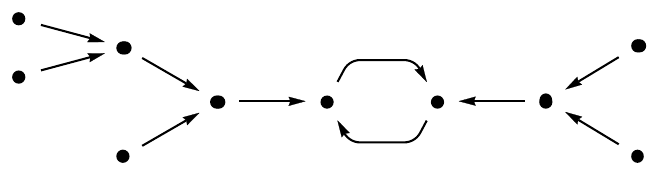}\\ \hline
\end{tabularx} \offinterlineskip
\begin{tabularx}{\textwidth}{|M|M|}
{\large$\boldsymbol{G_3}$} & {\large$\boldsymbol{G_4}$}
\end{tabularx} \offinterlineskip
\begin{tabularx}{\textwidth}{|W|W|}
\includegraphics[scale=.45]{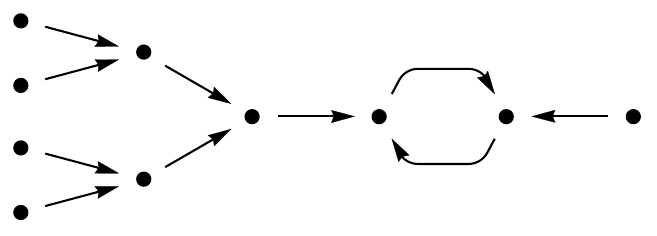} & \includegraphics[scale=.45]{graph1and2_3}\\ \hline
\end{tabularx} \offinterlineskip
\begin{tabularx}{\textwidth}{|M|M|}
{\large$\boldsymbol{G_5}$} & {\large$\boldsymbol{G_6}$}
\end{tabularx} \offinterlineskip
\begin{tabularx}{\textwidth}{|W|W|}
\includegraphics[scale=.43]{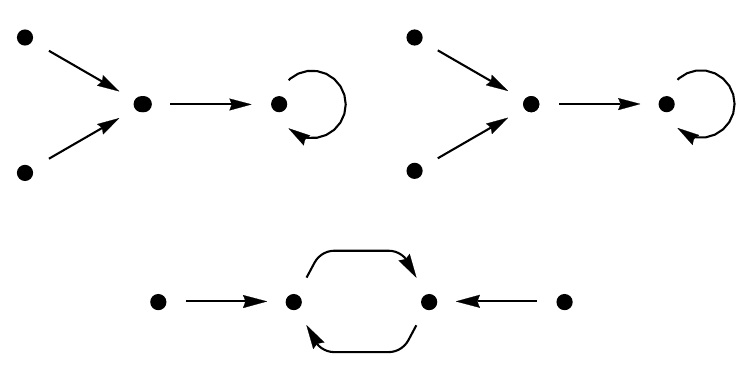} & \includegraphics[scale=.45]{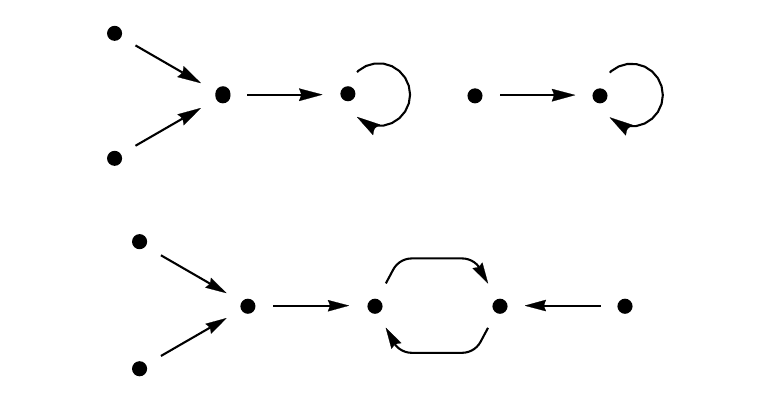}\\ \hline
\end{tabularx} \offinterlineskip 
\end{table}

\mbox{}
\newpage
\mbox{}
\newpage
\vfill
\mbox{}
\newpage
\section{Models for dynamical modular curves of genus 1 and 2}\label{app:models}
In Table~\ref{tab:models}, we give explicit models for the curves $X_1(G)$ discussed in Sections~\ref{sec:gen1} and \ref{sec:gen2}. All of these models were originally constructed in \cite{poonen:1998}, with the exception of the model for 8(4), which is given in \cite{morton:1998}. We refer the reader to those two articles for more information. A general point on $X_1(G)$ corresponds to a map $f_c$ together with a collection of preperiodic points for $f_c$, as described in \textsection\ref{sub:admissible}; in addition to giving a model for $X_1(G)$, we also express the parameter $c$ as a rational map on the given model.

\begin{table}[h]
	\caption{Models for dynamical modular curves of genus $1$ and $2$}
	\label{tab:models}
	\begin{tabular}{|l|l|l|}
	\hline
	Graph $G$ & Model for $X_1(G)$ & Parametrization of $c$\\\hline
		&& \\
	8(1,1)a & $y^2 = -(x^2 - 3)(x^2 + 1)$ & $c = -\dfrac{2(x^2 + 1)}{(x+1)^2(x - 1)^2}$ \\
		&& \\
	8(1,1)b & $y^2 = 2(x^3 + x^2 - x + 1)$ & $c = -\dfrac{2(x^2 + 1)}{(x+1)^2(x - 1)^2}$ \\
		&& \\
	8(2)a & $y^2 = 2(x^4 + 2x^3 - 2x + 1)$ & $c = -\dfrac{x^4 + 2x^3 + 2x^2 - 2x + 1}{(x+1)^2(x - 1)^2}$ \\
		&& \\
	8(2)b & $y^2 = 2(x^3 + x^2 - x + 1)$ & $c = -\dfrac{x^4 + 2x^3 + 2x^2 - 2x + 1}{(x+1)^2(x - 1)^2}$ \\
		&& \\
	10(2,1,1)a & $y^2 = 5x^4 - 8x^3 + 6x^2 + 8x + 5$ & $c = -\dfrac{(3x^2 + 1)(x^2 + 3)}{4(x+1)^2(x - 1)^2}$ \\
		&& \\
	10(2,1,1)b & $y^2 = (5x^2 - 1)(x^2 + 3)$ & $c = -\dfrac{(3x^2 + 1)(x^2 + 3)}{4(x+1)^2(x - 1)^2}$ \\\hline
		&& \\
	8(3) & $y^2 = x^6 - 2x^4 + 2x^3 + 5x^2 + 2x + 1$ & $c = -\frac{x^6 + 2x^5 + 4x^4 + 8x^3 + 9x^2 + 4x + 1}{4x^2(x+1)^2}$ \\
		&& \\
	8(4) & $y^2 = -x(x^2 + 1)(x^2 - 2x - 1)$ & $c = \frac{(x^2 - 4x - 1)(x^4 + x^3 + 2x^2 - x + 1)}{4x(x+1)^2(x-1)^2}$\\
		&& \\
	10(3,1,1) & $y^2 = x^6 + 2x^5 + 5x^4 + 10x^3 + 10x^2 + 4x + 1$ & $c = -\frac{x^6 + 2x^5 + 4x^4 + 8x^3 + 9x^2 + 4x + 1}{4x^2(x+1)^2}$\\
		&& \\
	10(3,2) & $y^2 = x^6 + 2x^5 + x^4 + 2x^3 + 6x^2 + 4x + 1$ & $c = -\frac{x^6 + 2x^5 + 4x^4 + 8x^3 + 9x^2 + 4x + 1}{4x^2(x+1)^2}$ \\
		&& \\\hline
	\end{tabular}
\end{table}

\vfill

\bibliography{C:/Users/John/Dropbox/jdoyle} 

\bibliographystyle{amsplain}

\end{document}